\documentclass[reqno,11pt]{amsart}
\usepackage[margin=1in]{geometry}
\usepackage{amsmath,amssymb,amsthm,amsfonts,mathrsfs}
\usepackage{color}
\usepackage{cite}
\usepackage{enumitem}
\usepackage[colorlinks=true,linkcolor=blue,citecolor=blue,urlcolor=blue]{hyperref}
\usepackage{tikz}
\usetikzlibrary{arrows,calc}
\newtheorem {theorem}{Theorem}[section]
\newtheorem {lemma}[theorem]{{\bf Lemma}}

\theoremstyle{remark}
\newtheorem {remark}{{\bf Remark}}[section]

\theoremstyle{plain} \numberwithin {equation}{section}
\def\nn{\nonumber}

\newcommand{\R}{\mathbb{R}}
\newcommand{\T}{\mathbb{T}}
\newcommand{\dd}{\,\mathrm{d}}

\begin{document}
\title[Stability of NS system]{Enhanced Dissipation and Stability Threshold for the Navier--Stokes Equations near Poiseuille Flow}

\author[T. Liang]{Tao Liang}
\address[Tao Liang]{\newline   School of Mathematics,
	South China University of Technology,
	Guangzhou, 510640, China}
\email{taolmath@163.com}

\author[C. L. Zhai]{Cuili Zhai}
\address[Cuili Zhai]{\newline   School of Mathematics and Physics
University of Science and Technology Beijing
100083, Beijing, People's Republic of China}
\email{zhaicuili035@126.com}

\author[X. P. Zhai]{Xiaoping Zhai}
\address[Xiaoping Zhai]{\newline   School of Mathematics and Statistics, Guangdong University of Technology,
	Guangzhou, 510520, China}
\email{pingxiaozhai@163.com (Corresponding author)}

\date{\today}

\maketitle

\begin{abstract}
We study the nonlinear stability of the quadratic plane Poiseuille flow
$U(y)=(y^2,0)$ for the two-dimensional incompressible Navier--Stokes equations on
$\mathbb T\times\mathbb R$. We prove that for initial perturbations of size $O(\nu^{2/3})$ in a $\nu$-dependent weighted Sobolev space, the flow is globally nonlinearly stable. In particular, the nonzero streamwise modes experience enhanced dissipation and decay exponentially on the characteristic time scale $O(\nu^{-1/2})$, while the zero mode remains uniformly controlled. Thus, in this weighted Sobolev topology, \(O(\nu^{2/3})\) is a sufficient stability scale without logarithmic loss.

\vspace{2mm}

\vspace{2mm}
\end{abstract}

\tableofcontents

\section{Introduction and main result}\label{sec:introduction}
\subsection{Background and previous results}
We consider the 2D incompressible Navier-Stokes (NS) equations on  $(x,y) \in \T \times \R$:
\begin{eqnarray}\label{model1}
\left\{\begin{aligned}
&\partial_t v - \nu \Delta v + (v \cdot \nabla)v + \nabla q = 0, \\
&\nabla \cdot v = 0,\\
&v|_{t=0} = v_{\mathrm{in}}(x, y),
\end{aligned}\right.
\end{eqnarray}
where $\nu = {Re}^{-1} > 0$ denotes the viscosity coefficient (inversely proportional to the Reynolds number), $v(t, x, y) \in \mathbb{R}^2$ is the velocity field, and $q \in \mathbb{R}$ is the pressure.

The profile $v_s = (y^2, 0)$ is a steady solution of \eqref{model1} provided that the pressure satisfies $\nabla Q \equiv (2\nu, 0)$. To investigate its stability, we introduce the velocity perturbation $u = v - v_s = (u_1, u_2)$ and the pressure perturbation $p = q - Q$. The governing equations for the perturbation are:
\begin{eqnarray}\label{model2}
\left\{\begin{aligned}
&\partial_t u - \nu \Delta u + (u \cdot \nabla)u+ y^2\partial_x u + \begin{pmatrix} 2yu_2 \\ 0 \end{pmatrix} + \nabla p = 0, \\
&\nabla \cdot u = 0,\\
&u|_{t=0} = u_{\mathrm{in}}(x, y).
\end{aligned}\right.
\end{eqnarray}
Introducing a stream function $\psi$ through $u=\nabla^\perp\psi=(-\partial_y\psi,\partial_x\psi)$, and defining  $\omega=\Delta\psi$, we obtain
\begin{eqnarray}\label{wholeNS}
\left\{\begin{aligned}
&\partial_t\omega+y^2\partial_x\omega-2\partial_x\psi-\nu\Delta\omega=-u\cdot\nabla\omega,\\
&\omega=\Delta\psi,
\qquad
u=\nabla^\perp\psi,\qquad\qquad\qquad\quad\qquad
(x,y)\in\mathbb T\times\mathbb R.
\end{aligned}\right.
\end{eqnarray}
In this formulation, the term $y^2\partial_x\omega$ dictates phase mixing, while $-2\partial_x\psi$ represents a nonlocal coupling induced by the curvature of the background shear. This curvature term is the main feature distinguishing the linearized equation from a passive-scalar equation.

\medskip
Hydrodynamic stability at high Reynolds numbers has been a central theme in fluid dynamics since Reynolds's seminal 1883 experiment \cite{Re}. While certain laminar flows are linearly stable at all Reynolds numbers \cite{DR,Rom}, they often exhibit nonlinear instability under finite-amplitude perturbations as $\nu \to 0$ \cite{DHB,TA}. As first observed by Kelvin \cite{Kel}, the basin of attraction for these laminar flows shrinks as the Reynolds number increases. This observation motivates the so-called transition threshold problem, classically conceptualized by Trefethen et al.~\cite{T} and rigorously formulated by Bedrossian, Germain, and Masmoudi \cite{BGM2017}: given a function space $X$, determine the critical exponent $\gamma=\gamma(X)$ such that
\begin{align*}
		&\|u_{\mathrm{in}}\|_{X}\ll \nu^{\gamma} \Longrightarrow\ \mathrm{stability},\\
		&\|u_{\mathrm{\mathrm{in}}}\|_{X}\gg \nu^{\gamma}\Longrightarrow\ \mathrm{instability}.
\end{align*}
For Couette flow on \(\mathbb T\times\mathbb R\), the following threshold bounds are known:
\begin{itemize}[leftmargin=*]
	\item In Gevrey class $2^-$, $\gamma\leq 0$ \cite{BMV2016}, while $\gamma\geq 0$ was established in \cite{DM2023}.
	\item In the Sobolev space $H^{\log}_xL^2_y$, $\gamma\leq \tfrac12$ \cite{BVW2018,zhaoweiren2020cpde}, with the lower bound $\gamma\geq \tfrac12$ proved in \cite{LiMasmoudiZhao2022critical}.
	\item In $H^\sigma$ ($\sigma\geq 2$), $\gamma\leq \tfrac13$ \cite{zhaoweiren2019,wei2023}.
	\item In Gevrey class $\tfrac1s$ ($s\in[0,\tfrac12]$), $\gamma\leq \tfrac{1-2s}{3(1-s)}$ \cite{LMZ2022G}.
\end{itemize}
Analogous threshold results have been established for 2D Couette flow in finite channels \cite{BHIW2023,BHIW,CLWZ2020} and for 3D Couette flow \cite{BGM2017,BGM2020,BGM2015,Chenwei2020,wei2020}.

For the Poiseuille profile \(U(y)=y^2\), \(U'(0)=0\) and \(U''(0)=2\). Hence phase mixing degenerates near the critical point \(y=0\).  Balancing quadratic phase mixing with viscous diffusion yields mode-dependent length and time scales:
\begin{align*}
t_{\mathrm{ED},k}\sim(\nu|k|)^{-1/2},
\qquad
\lambda_{\mathrm{ED},k}\sim(\nu|k|)^{1/2},
\qquad k\neq0.
\end{align*}
Consequently, the lowest nonzero streamwise modes decay on the enhanced-dissipation time scale $O(\nu^{-1/2})$, which is slower than the $O(\nu^{-1/3})$ Couette scale but significantly faster than standard viscous diffusion.

Previous studies of Poiseuille-type flows have used resolvent estimates, hypocoercivity, and inviscid damping. For the 2D quadratic shear on $\mathbb T\times\mathbb R$, Coti Zelati, Elgindi, and Widmayer \cite{CotiZelati2020a} proved enhanced dissipation and nonlinear stability for perturbations of size $\nu^{3/4+}$ in a weighted vorticity space. Del Zotto \cite{DelZotto2023} subsequently refined the hypocoercive
method and obtained nonlinear stability under an $L^2$ smallness
condition of size
$\nu^{2/3}(1+|\log\nu|^{1/2})^{-1}$, together with enhanced
dissipation of the nonzero streamwise modes on the time scale
$O(\nu^{-1/2})$. In the bounded domain setting with Navier-slip boundaries, Ding and Lin established a $\nu^{3/4}$ threshold via resolvent estimates \cite{Ding2022}, later improving it to $\nu^{2/3}$ \cite{Ding2025}.

Under no-slip boundary conditions, Chen, Li, Shen, and Zhang \cite{ChenLiShenZhang2025} recently analyzed a class of symmetric shear flows in $\mathbb T\times(-1,1)$ (including the quadratic profile). By simultaneously controlling the inviscid-damping norm $|\alpha|^{1/2}\|u_\alpha\|_{L_{t,y}^2}$ and the enhanced-dissipation norm $\nu^{1/4}|\alpha|^{1/4}\|\omega_\alpha\|_{L_{t,y}^2}$, they obtained a sufficient $O(\nu^{2/3})$ stability scale. In 3D, Chen, Ding, Lin, and Zhang \cite{ChenDingLinZhang} proved global stability for perturbations of size $\nu^{7/4}$ in a finite channel.

\subsection{Main result}
For a function $f=f(x,y)$, we denote its streamwise average and the projection onto nonzero streamwise modes by
\begin{equation*}
P_0f=f_0
:=
\frac{1}{2\pi}
\int_{-\pi}^{\pi}f(x,y)\,\mathrm dx,
\qquad
f_{\neq}:=f-f_0.
\end{equation*}
We define the Fourier multiplier $\langle\partial_x\rangle = (1-\partial_x^2)^{1/2}$. Unless otherwise specified, an $L^2$ norms of nonzero modes are taken over $\mathbb T\times\mathbb R$, while those of streamwise averages are defined over $\mathbb R$.

Our main result is stated as follows.
\begin{theorem}\label{thm1}
Fix $m>9/8$. There exist constants $\varepsilon_0=\varepsilon_0(m)>0$, $\delta_0=\delta_0(m)>0$, and $C=C(m)>0$, all independent of $\nu\in(0,1)$, such that the following holds.

Let $u_{\mathrm{in}}=\nabla^\perp\psi_{\mathrm{in}}$ be an initial velocity perturbation with corresponding vorticity $\omega_{\mathrm{in}}=\Delta\psi_{\mathrm{in}}$. Suppose that $0<\varepsilon\leq\varepsilon_0$ and the initial data satisfy
\begin{align}\label{initial_data}
\notag\|u_{\mathrm{in},0}\|_{L_y^2}^2
+&\nu^{1/3}
 \|\partial_y\omega_{\mathrm{in},0}\|_{L_y^2}^2  + \| \langle \partial_x\rangle^{m}\omega_{\mathrm{in,\neq}}\|_{L_{x,y}^2}^2
+\nu^{1/2}
 \big\|
 \nabla|\partial_x|^{m-\frac14}
 \omega_{\mathrm{in},\neq}
 \big\|_{L_{x,y}^2}^2
\\
&+\nu^{-1/2}
 \bigl\|
 |\partial_x|^{m+1/4}
 (y\omega_{\mathrm{in},\neq})
 \bigr\|_{L_{x,y}^2}^2
+\nu^{-\frac12}
 \bigl\|
 |\partial_x|^{m+1/4}
 u_{\mathrm{in},\neq}
 \bigr\|_{L_{x,y}^2}^2
\leq
\varepsilon^2\nu^{4/3}.
\end{align}
Then the associated unique solution to \eqref{wholeNS} exists globally in time. Moreover, for all $t\geq0$, the nonzero modes satisfy
\begin{equation}\label{nonzero}
\big\|
\langle\partial_x\rangle^{m}
\omega_{\neq}(t)
\big\|_{L_{x,y}^2}+\left\|
\langle\partial_x\rangle^{m-{1}/{4}}
\nu^{1/4}\partial_y
\omega_{\neq}(t)
\right\|_{L_{x,y}^2}
\leq
C\varepsilon\nu^{2/3}
e^{-\delta_0\nu^{1/2}t},
\end{equation}
and the zero mode satisfies
\begin{equation}\label{zero}
\|\omega_0(t)\|_{L_y^2}
+\nu^{1/6}
 \|\partial_y\omega_0(t)\|_{L_y^2}
\leq
C\varepsilon\nu^{2/3}.
\end{equation}
In particular, the nonzero streamwise modes experience enhanced dissipation, decaying exponentially on the characteristic time scale $t\sim\nu^{-1/2}$.
\end{theorem}
\begin{remark}
Del Zotto \cite{DelZotto2023} established nonlinear stability under an $L^2$ condition $\|\omega_{\mathrm{in}}\|_{L^2}\lesssim
\nu^{2/3}(1+|\log\nu|^{1/2})^{-1}$. In contrast, \eqref{initial_data}
imposes stronger, $\nu$-dependent weighted Sobolev assumptions, including
control of streamwise derivatives, $y\omega_{\mathrm{in},\neq}$, and the
nonzero-mode velocity. Thus, Theorem~\ref{thm1} is not a logarithm-free
improvement for arbitrary $L^2$ data. Rather, it provides a sufficient
stability criterion at scale $\nu^{2/3}$ without logarithmic loss in this
stronger topology, together with weighted enhanced-dissipation estimates
and simultaneous control of the zero and nonzero streamwise modes.
\end{remark}
\begin{remark}\label{rem:restriction_m}
The condition $m>\frac98$ arises from the summability requirement
$\sum_{\ell\neq0}|\ell|^{5/4-2m}<\infty$ used in \eqref{psi_neq3}.
It is sufficient for the present weighted energy argument but is not
claimed to be optimal; the endpoint $m=\frac98$ is not addressed here.
\end{remark}

\subsection{Organization of the paper}\label{subsec:organization}

Section~\ref{sec:linearized-stability} develops the modewise hypocoercive energy method and proves the corresponding coercivity and enhanced-dissipation estimates.
Section~\ref{sec:preliminary-lemmas} collects the Sobolev, interpolation, and weighted elliptic estimates needed below.
In Section~\ref{sec:nonlinear-stability}, we derive the zero- and nonzero-mode equations, state the nonlinear estimates, and close the global bootstrap argument.
The detailed proofs of these estimates, based on mean--fluctuation decompositions and a suitable frequency partition, are given in Section~\ref{sec:proof-nonlinear-lemmas}.

\section{Linearized Stability}\label{sec:linearized-stability}
We begin by establishing the modewise hypocoercive estimates that govern the enhanced dissipation. Let $f_k(t,y)$ denote the partial Fourier coefficients in the streamwise variable $x$, defined by
\begin{equation*}
f_k(t,y) = \frac{1}{2\pi}\int_{-\pi}^{\pi} f(t,x,y)e^{-ikx} \dd x, \qquad k\in\mathbb{Z}.
\end{equation*}
We introduce the mode-dependent differential operators
\begin{equation*}
\nabla_k = (ik,\partial_y), \qquad \Delta_k = \partial_y^2 - k^2.
\end{equation*}
For each $k\neq0$, the linearization of \eqref{wholeNS} around the Poiseuille flow reduces to the following system on $\mathbb{R}$:
\begin{equation}\label{linear_mode}
\partial_t\omega_k + iky^2\omega_k - 2ik\psi_k - \nu\Delta_k\omega_k = 0, \qquad \omega_k = \Delta_k\psi_k.
\end{equation}

The following identities are used to construct the modewise hypocoercive energy. Here and in the sequel, $\langle \cdot, \cdot \rangle$ denotes the standard $L^2$ inner product over $y \in \mathbb{R}$.

\begin{lemma}\label{Zotto}
For any $k\neq0$, sufficiently smooth solutions to \eqref{linear_mode} satisfy the following four identities.
\begin{align}
	&\frac{1}{2}\frac{\mathrm{d}}{\mathrm{d}t}\|\omega_k\|_{L^2}^2 + \nu\|\nabla_k \omega_k\|_{L^2}^2 = 0,
\label{L2}\\
	&\frac{1}{2}\frac{\mathrm{d}}{\mathrm{d}t}\|\nabla_k \omega_k\|_{L^2}^2 + \nu\|\Delta_k \omega_k\|_{L^2}^2
	= -2\operatorname{Re} \langle iky \omega_k, \partial_y \omega_k \rangle,\label{H1}\\
	&\frac{1}{2}\frac{\mathrm{d}}{\mathrm{d}t}\left(\|iky \omega_k\|_{L^2}^2 + 2\|ik\nabla_k \psi_k\|_{L^2}^2 \right)
	+ \nu\|ik \omega_k\|_{L^2}^2 + \nu\|iky\nabla_k \omega_k\|_{L^2}^2 = 0,\label{ED}
\\
	&\frac{\mathrm{d}}{\mathrm{d}t}\operatorname{Re}\langle \partial_y \omega_k, iky\omega_k \rangle
	+ 2\|iky \omega_k\|_{L^2}^2 + 4\|ik\partial_y \psi_k\|_{L^2}^2
	= -2\nu \operatorname{Re} \langle \Delta_k \omega_k, iky\partial_y \omega_k \rangle.\label{ID}
\end{align}
\end{lemma}
\begin{proof}
This lemma follows directly from Proposition 2.3 in \cite{DelZotto2023}; we omit the proof for brevity.
\end{proof}
We next combine \eqref{ID} with an additional estimate to obtain coercive control of $y\omega_k$, $\partial_y\psi_k$, and $\psi_k$.
\begin{lemma}\label{le_ID}
For every $k\neq0$, the mixed energy satisfies
	\begin{align}\label{molify_ID}
	\frac{\mathrm{d}}{\mathrm{d}t}\operatorname{Re}\langle \partial_y \omega_k,\, iky\omega_k \rangle
	+ \|iky \omega_k\|_{L^2}^2 + 2|k|^2\|\partial_y \psi_k\|_{L^2}^2 + |k|^4 \| \psi_k\|^2_{L^2}
	\le -2\nu\operatorname{Re}\langle \Delta_k \omega_k,\, iky\partial_y \omega_k \rangle .
	\end{align}
\end{lemma}
\begin{proof}
	Differentiating the mixed inner product and separating the contributions from transport, elliptic coupling, and viscous dissipation yields
	\begin{align*}
		\frac{\mathrm{d}}{\mathrm{d}t} \operatorname{Re} \left\langle \partial_y \omega_k, i k y \omega_k \right\rangle
		&= \operatorname{Re} \int_{\mathbb{R}} \overline{\partial_y \omega_k}\, (i k y \partial_t \omega_k) \,\mathrm{d}y
		+ \operatorname{Re} \int_{\mathbb{R}} (\partial_y \partial_t \omega_k)\, \overline{i k y \omega_k} \,\mathrm{d}y \\
		&= \operatorname{Re} \int_{\mathbb{R}} \overline{\partial_y \omega_k}\, i k y
		\big( - i k y^2 \omega_k + 2 i k \psi_k + \nu(\partial_y^2-k^2)\omega_k \big) \,\mathrm{d}y \\
		&\quad + \operatorname{Re} \int_{\mathbb{R}} \big( - 2i k y \omega_k - i k y^2 \partial_y \omega_k + 2 i k \partial_y \psi_k + \nu(\partial_y^2-k^2)\partial_y \omega_k \big) \overline{i k y \omega_k} \,\mathrm{d}y \\
		&=: I_1 + I_2.
	\end{align*}
	
	We expand $I_1 = I_{1,1} + I_{1,2} + I_{1,3}$ and evaluate each term via integration by parts:
	\begin{align*}
		I_{1,1} &= \operatorname{Re} \int_{\mathbb{R}} k^2 y^3 \overline{\partial_y \omega_k}\,\omega_k \,\mathrm{d}y
		= \frac{1}{2} \int_{\mathbb{R}} k^2 y^3 \partial_y |\omega_k|^2 \,\mathrm{d}y
		= -\frac{3}{2} \|iky \omega_k\|_{L^2}^2, \\
		I_{1,2} &= -2\operatorname{Re} \int_{\mathbb{R}} k^2y \overline{\partial_y \omega_k}\,\psi_k \,\mathrm{d}y
		= 2\operatorname{Re} \int_{\mathbb{R}} k^2\overline{\omega_k}\psi_k \,\mathrm{d}y + 2\operatorname{Re} \int_{\mathbb{R}} k^2y\overline{\omega_k}\partial_y\psi_k \,\mathrm{d}y \\
		&= -2\|k\nabla_k\psi_k\|_{L^2}^2 + 2\operatorname{Re} \int_{\mathbb{R}} k^2y\overline{\omega_k}\partial_y\psi_k \,\mathrm{d}y, \\
		I_{1,3} &= -\nu \operatorname{Re} \int_{\mathbb{R}} \Delta_k \omega_k \, \overline{ik y \partial_y \omega_k} \,\mathrm{d}y
		= -\nu \operatorname{Re} \left\langle \Delta_k \omega_k, iky\partial_y \omega_k \right\rangle.
	\end{align*}
	
	Similarly, decomposing $I_2 = I_{2,1} + I_{2,2} + I_{2,3} + I_{2,4}$, we find
	\begin{align*}
		I_{2,1} &= -2\operatorname{Re} \int_{\mathbb{R}} k^2y^2|\omega_k|^2 \,\mathrm{d}y = -2\|iky\omega_k\|_{L^2}^2, \\
		I_{2,2} &= \operatorname{Re} \int_{\mathbb{R}} (-iky^2\partial_y \omega_k) \overline{(iky\omega_k)} \,\mathrm{d}y = \frac{3}{2} \|iky\omega_k\|_{L^2}^2, \\
		I_{2,3} &= 2\operatorname{Re} \int_{\mathbb{R}} k^2y\,\overline{\omega_k}\,\partial_y\psi_k \,\mathrm{d}y, \\
		I_{2,4} &= \nu \operatorname{Re}\left\langle \Delta_k\partial_y \omega_k, iky\omega_k \right\rangle
		= -\nu \operatorname{Re}\left\langle \Delta_k \omega_k, ik\omega_k \right\rangle - \nu \operatorname{Re}\left\langle \Delta_k \omega_k, iky\partial_y \omega_k \right\rangle \\
		&= -\nu \operatorname{Re}\left\langle \Delta_k \omega_k, iky\partial_y \omega_k \right\rangle,
	\end{align*}
	where we used the fact that $\operatorname{Re}\langle \Delta_k \omega_k, ik\omega_k \rangle = 0$ due to the self-adjointness of $\Delta_k$.
	
	Summing the contributions from $I_1$ and $I_2$, we obtain
	\begin{align}\label{jiaj1}
		\frac{\mathrm{d}}{\mathrm{d}t} \operatorname{Re} \left\langle \partial_y \omega_k, iky \omega_k \right\rangle
		=& -2\|k\nabla_k\psi_k\|_{L^2}^2 - 2\|ik y \omega_k\|_{L^2}^2
		 \nn\\
		&+ 4\operatorname{Re} \int_{\mathbb{R}} k^2y\,\overline{\omega_k}\,\partial_y\psi_k \,\mathrm{d}y- 2\nu \operatorname{Re} \left\langle \Delta_k \omega_k, iky\partial_y \omega_k \right\rangle.
	\end{align}
	By applying Cauchy-Schwarz and Young's inequalities, the cross term is bounded by
	\begin{equation*}
		4\operatorname{Re} \int_{\mathbb{R}} k^2y\,\overline{\omega_k}\,\partial_y\psi_k \,\mathrm{d}y \le 2\left( \|iky\omega_k\|_{L^2}^2 + \|ik\partial_y\psi_k\|_{L^2}^2 \right).
	\end{equation*}
	Substituting this bound into \eqref{jiaj1} and observing the exact cancellation of the terms $-2\|ik\partial_y\psi_k\|_{L^2}^2$ (contained within $-2\|k\nabla_k\psi_k\|_{L^2}^2$) and $-2\|iky\omega_k\|_{L^2}^2$, we deduce
	\begin{equation}\label{modifyID}
		\frac{\mathrm{d}}{\mathrm{d}t} \operatorname{Re} \left\langle \partial_y \omega_k, iky \omega_k \right\rangle + 2|k|^4\| \psi_k\|_{L^2}^2
		\le -2\nu \operatorname{Re} \left\langle \Delta_k \omega_k, iky\partial_y \omega_k \right\rangle.
	\end{equation}
	Finally, averaging \eqref{ID} and \eqref{modifyID} yields the desired inequality \eqref{molify_ID}.
\end{proof}

We now construct the modewise hypocoercive energy functional. For each $k \neq 0$, we define
\begin{align}\label{Ek_def}
	E_k[\omega_k](t) =& \frac{1}{2} \|\omega_k\|_{L^2}^2 + \frac{A}{2}\|\nabla_k \omega_k\|_{L^2}^2
	\nn\\
&+ \frac{C}{2}\big(\|iky \omega_k\|_{L^2}^2 + 2\|ik\nabla_k \psi_k\|_{L^2}^2 \big)
	+ B\operatorname{Re}\langle \partial_y \omega_k, iky\omega_k \rangle,
\end{align}
where the scaling parameters are chosen to reflect the enhanced dissipation rates:
\begin{equation*}
	A = \alpha \nu^{1/2}|k|^{-1/2}, \qquad
	B = \beta |k|^{-1}, \qquad
	C = \gamma \nu^{-1/2}|k|^{-3/2},
\end{equation*}
for some absolute constants $\alpha, \beta, \gamma > 0$ to be determined.

Taking the corresponding linear combination of the identities established in Lemma \ref{Zotto}, and bounding the cross terms via the Cauchy--Schwarz and Young's inequalities, we can estimate the evolution of $E_k$:
\begin{align*}
	&\frac{\mathrm{d}}{\mathrm{d}t} E_k[\omega_k]
	+ \nu\|\nabla_k \omega_k\|_{L^2}^2
	+ \alpha\nu^{3/2}|k|^{-1/2} \|\Delta_k \omega_k\|_{L^2}^2+ \beta |k|\| y\omega_k \|_{L^2}^2 \\
	&\qquad
	+ 2\beta |k|\|\partial_y\psi_k \|_{L^2}^2
	+ \beta |k|^3\|\psi_k\|_{L^2}^2+ \gamma\nu^{1/2}|k|^{1/2} \|\omega_k\|_{L^2}^2
	+ \gamma\nu^{1/2}|k|^{1/2} \|y \nabla_k \omega_k\|_{L^2}^2 \\
	&\quad\le 2\alpha\nu^{1/2}|k|^{-1/2} \big|\langle iky\omega_k, \partial_y \omega_k \rangle\big|
	+ 2\nu\beta |k|^{-1} \big|\langle \Delta_k \omega_k, iky\partial_y \omega_k \rangle\big| \\
	&\quad\le \frac{3\nu}{4}\|\partial_y \omega_k\|_{L^2}^2 + \frac43 \alpha^2 |k|\|y \omega_k\|_{L^2}^2
	+ \frac{\alpha}{2}\nu^{3/2}|k|^{-1/2} \|\Delta_k \omega_k\|_{L^2}^2
	+ \frac{2\beta^2}{\alpha}\nu^{1/2}|k|^{1/2} \|y\partial_y \omega_k\|_{L^2}^2.
\end{align*}
Absorbing the terms from the right-hand side into the left-hand side yields the dissipation bound
\begin{equation}\label{Ek_dissipation}
	\begin{aligned}
		\frac{\mathrm{d}}{\mathrm{d}t}E_k[\omega_k]
		+& \frac{1}{4} \nu\|\nabla_k\omega_k\|_{L^2}^2
		+ \frac{\alpha}{2}\nu^{3/2}|k|^{-1/2} \|\Delta_k \omega_k\|_{L^2}^2 \\
		&+ (\beta-\frac43 \alpha^2)|k|\|y \omega_k\|_{L^2}^2
		+ 2\beta |k|\|\partial_y\psi_k\|_{L^2}^2
		+ \beta |k|^3\|\psi_k\|_{L^2}^2 \\
		&+ \gamma\nu^{1/2}|k|^{1/2}\|\omega_k\|_{L^2}^2
		+ \left( \gamma - \frac{2\beta^2}{\alpha} \right) \nu^{1/2}|k|^{1/2} \|y \nabla_k \omega_k\|_{L^2}^2
		\le 0.
	\end{aligned}
\end{equation}
To verify that the modewise functional $E_k[\omega_k]$ is coercive,
we estimate the mixed term in terms of the positive diagonal terms in
\eqref{Ek_def}.
Since $k\neq0$, we have
\begin{align*}
\|iky\omega_k\|_{L^2}^2
=|k|^2\|y\omega_k\|_{L^2}^2,
\qquad
\|ik\nabla_k\psi_k\|_{L^2}^2
=|k|^2\|\nabla_k\psi_k\|_{L^2}^2.
\end{align*}
By the Cauchy--Schwarz and Young inequalities,
\begin{align}\label{Ek_mixed_control}
\big|
B\operatorname{Re}
\langle \partial_y\omega_k,iky\omega_k\rangle
\big|
&\le
\frac{\beta^2}{\gamma}
\nu^{1/2}|k|^{-1/2}
\|\partial_y\omega_k\|_{L^2}^2
+
\frac{\gamma}{4}
\nu^{-1/2}|k|^{-3/2}
\|iky\omega_k\|_{L^2}^2
\nonumber\\
&\le
\frac{\beta^2}{\gamma}
\nu^{1/2}|k|^{-1/2}
\|\nabla_k\omega_k\|_{L^2}^2
+
\frac{\gamma}{4}
\nu^{-1/2}|k|^{-3/2}
\|iky\omega_k\|_{L^2}^2.
\end{align}
Consequently, using the definition \eqref{Ek_def}, we obtain the lower
bound
\begin{align}\label{Ek_lower_bound}
E_k[\omega_k]
\ge&
\frac12\|\omega_k\|_{L^2}^2
+
\left(
\frac{\alpha}{2}-\frac{\beta^2}{\gamma}
\right)
\nu^{1/2}|k|^{-1/2}
\|\nabla_k\omega_k\|_{L^2}^2
\nonumber\\
&
+\frac{\gamma}{4}
\nu^{-1/2}|k|^{1/2}
\|y\omega_k\|_{L^2}^2
+\gamma
\nu^{-1/2}|k|^{1/2}
\|\nabla_k\psi_k\|_{L^2}^2.
\end{align}
In particular, if
\begin{equation}\label{Ek_coercivity_condition}
\frac{\alpha}{2}-\frac{\beta^2}{\gamma}>0,
\end{equation}
then $E_k[\omega_k]$ controls all the positive terms appearing in
\eqref{Ek_def}.

The reverse estimate follows from \eqref{Ek_mixed_control}. Namely,
\begin{align}\label{Ek_upper_bound}
E_k[\omega_k]
\le&
\frac12\|\omega_k\|_{L^2}^2
+
\left(
\frac{\alpha}{2}+\frac{\beta^2}{\gamma}
\right)
\nu^{1/2}|k|^{-1/2}
\|\nabla_k\omega_k\|_{L^2}^2
\nonumber\\
&
+\frac{3\gamma}{4}
\nu^{-1/2}|k|^{1/2}
\|y\omega_k\|_{L^2}^2
+\gamma
\nu^{-1/2}|k|^{1/2}
\|\nabla_k\psi_k\|_{L^2}^2.
\end{align}
Combining \eqref{Ek_lower_bound} and \eqref{Ek_upper_bound}, we conclude
that
\begin{equation}\label{Ekdengjia}
E_k[\omega_k]
\approx
\|\omega_k\|_{L^2}^2
+\nu^{1/2}|k|^{-1/2}\|\nabla_k\omega_k\|_{L^2}^2
+\nu^{-1/2}|k|^{1/2}
\left(
\|y\omega_k\|_{L^2}^2
+\|\nabla_k\psi_k\|_{L^2}^2
\right),
\end{equation}
where the implicit constants are uniform with respect to
$\nu\in(0,1)$ and $k\neq0$.

We now fix
$
(\alpha,\beta,\gamma)=(1,2,9).
$
For this choice,
\begin{align*}
\frac{\alpha}{2}-\frac{\beta^2}{\gamma}
=\frac1{18}>0,
\qquad
\beta-\frac43\alpha^2=\frac23>0,
\qquad
\gamma-\frac{2\beta^2}{\alpha}=1>0,
\end{align*}
and
\begin{align*}
\frac{2\beta^2}{\alpha\gamma}
=\frac89<1.
\end{align*}

Hence both the uniform equivalence \eqref{Ekdengjia} and the
positivity of all the dissipation coefficients in
\eqref{Ek_dissipation} hold for every $\nu\in(0,1)$ and every
$k\neq0$. In particular, after possibly decreasing the universal
constant $c>0$, \eqref{Ek_dissipation} implies
\begin{equation}\label{Ek_dissipation_coercive}
\frac{\mathrm{d}}{\mathrm{d}t}E_k[\omega_k]
+c\sum_{j=1}^{5}\mathcal D_{j,k}[\omega_k]
\le0,
\end{equation}
where the constituent dissipation terms are defined as
\begin{equation*}
	\begin{alignedat}{2}
		\mathcal{D}_{1,k}[\omega_k] &:= \nu\|\nabla_k \omega_k\|_{L^2}^2, &\qquad
		\mathcal{D}_{2,k}[\omega_k] &:= \nu^{3/2}|k|^{-1/2}\|\Delta_k \omega_k\|_{L^2}^2,\\
		\mathcal{D}_{3,k}[\omega_k] &:= |k|\|y\omega_k\|_{L^2}^2 + |k|\|\nabla_k\psi_k\|_{L^2}^2, &\qquad
		\mathcal{D}_{4,k}[\omega_k] &:= \nu^{1/2}|k|^{1/2}\|\omega_k\|_{L^2}^2,\\
		\mathcal{D}_{5,k}[\omega_k] &:= \nu^{1/2}|k|^{1/2} \|y\nabla_k \omega_k\|_{L^2}^2. & &
	\end{alignedat}
\end{equation*}

\section{Preliminary Lemmas}\label{sec:preliminary-lemmas}
Before bounding the nonlinear interactions, we next define the energy and dissipation functionals used in the bootstrap argument.
\begin{equation}\label{def_E0}
	\mathcal{E}_0[\omega_0] = \|\omega_0\|_{L^2}^2 + \alpha \nu^{1/3}\|\partial_y \omega_0\|_{L^2}^2,
	\qquad
	\mathcal{D}_0[\omega_0] = \nu\|\partial_y \omega_0\|_{L^2}^2 + \alpha \nu^{4/3}\|\partial_y^2 \omega_0\|_{L^2}^2,
\end{equation}
where $\alpha > 0$ is a structural parameter to be determined. To capture the enhanced dissipation of the fluctuating modes, we introduce the exponentially weighted functionals
\begin{equation}\label{def_Eneq}
	\mathcal{E}_{\neq}[\omega_k] = \sum_{k\neq 0} e^{2\delta_0 \nu^{1/2}t} |k|^{2m} E_k[\omega_k],
	\qquad
	\mathcal{D}_{\neq}[\omega_k] = \sum_{j=1}^5 \mathcal{D}_{\neq,j}[\omega_k],
\end{equation}
where the individual dissipation components are given by
\begin{equation*}
	\mathcal{D}_{\neq,j}[\omega_k] = \sum_{k\neq 0} e^{2\delta_0 \nu^{1/2}t} |k|^{2m} \mathcal{D}_{j,k}[\omega_k], \qquad j \in \Big\{1, 2, 3, 4, 5\Big\}.
\end{equation*}
We define the total energy and dissipation by \begin{equation*}
	\mathcal{E}[\omega] := \mathcal{E}_0[\omega_0] + \mathcal{E}_{\neq}[\omega_k], \qquad \mathcal{D}[\omega] := \mathcal{D}_0[\omega_0] + \mathcal{D}_{\neq}[\omega_k].
\end{equation*}
When no confusion can arise, we suppress the arguments of these
functionals and simply write $\mathcal{E}_0$, $\mathcal{D}_0$,   $\mathcal{E}_{\neq}$,
 $\mathcal{D}_{\neq}$,   $\mathcal{D}_{\neq,j}$, $\mathcal{E}$, and $\mathcal{D}$.

In the subsequent lemmas, we collect several interpolation, Sobolev, and elliptic estimates used in the nonlinear estimates.

\begin{lemma}\label{psi_0}
	The zero-mode velocity and vorticity satisfy the following pointwise $L^\infty_y$ bounds:
	\begin{align}
		\|\partial_y\psi_0\|_{L^\infty} &\lesssim \|\partial_y\psi_0\|_{L^2}^{1/2}\|\omega_0\|_{L^2}^{1/2}, \label{zero_velocity_Linfty}\\
		\|\omega_0\|_{L^\infty} &\lesssim \nu^{-1/12}\mathcal{E}_0^{1/2}. \label{zero_vorticity_Linfty}
	\end{align}
\end{lemma}

\begin{proof}
	The first estimate \eqref{zero_velocity_Linfty} is a direct consequence of the one-dimensional Gagliardo--Nirenberg interpolation inequality, utilizing the identity $\partial_y^2\psi_0 = \omega_0$.
	
	Applying the same interpolation inequality to the zero-mode vorticity yields
	\begin{equation*}
		\|\omega_0\|_{L^\infty} \lesssim \|\omega_0\|_{L^2}^{1/2}\|\partial_y\omega_0\|_{L^2}^{1/2}.
	\end{equation*}
	Recalling the definition of the zero-mode energy \eqref{def_E0}, we have the uniform bounds $\|\omega_0\|_{L^2} \le \mathcal{E}_0^{1/2}$ and $\|\partial_y\omega_0\|_{L^2} \le \alpha^{-1/2} \nu^{-1/6}\mathcal{E}_0^{1/2}$. Substituting these into the interpolation bound  gives \eqref{zero_vorticity_Linfty}, and the proof is complete.
\end{proof}
\begin{lemma}\label{w_neq}
	The following interpolation estimates hold for the nonzero modes:
	\begin{align}
		\sum_{k\neq 0}|k|^{2m+1}\|\omega_k\|_{L^2}^2
		&\lesssim \nu^{-2/3} e^{-2\delta_0\nu^{1/2}t} \mathcal{D}_{\neq,1}^{1/3} \mathcal{D}_{\neq,4}^{2/3}, \label{w1} \\
		\sum_{k\neq 0} |k|^{2m+\frac{3}{2}}\|\omega_k\|_{L^2}^2
		&\lesssim \nu^{-5/6} e^{-2\delta_0\nu^{1/2}t} \mathcal{D}_{\neq,1}^{2/3} \mathcal{D}_{\neq,4}^{1/3}, \label{w2} \\
		\sum_{k\neq 0}\|\omega_k\|_{L^\infty}
		&\lesssim \nu^{-1/8} e^{-\delta_0\nu^{1/2}t} \mathcal{E}_{\neq}^{1/2}. \label{w3}
	\end{align}
\end{lemma}

\begin{proof}
	To establish \eqref{w1}, we  decompose the Fourier weight $|k|^{2m+1}$ into a convex combination of the weights naturally appearing in $\mathcal{D}_{\neq,1}$ and $\mathcal{D}_{\neq,4}$, namely $|k|^{2m+2}$ and $|k|^{2m+1/2}$. Applying the discrete H\"older inequality yields
	\begin{align*}
		\sum_{k\neq 0} |k|^{2m+1}\|\omega_k\|_{L^2}^2
		\le& \sum_{k\neq 0} \left( |k|^{2m+2}\|\omega_k\|_{L^2}^2 \right)^{1/3} \left( |k|^{2m+1/2}\|\omega_k\|_{L^2}^2 \right)^{2/3} \\
		\le& \Bigg( \sum_{k\neq 0} |k|^{2m+2}\|\omega_k\|_{L^2}^2 \Bigg)^{1/3} \Bigg( \sum_{k\neq 0} |k|^{2m+1/2}\|\omega_k\|_{L^2}^2 \Bigg)^{2/3}\\
\lesssim& \nu^{-2/3}
	\mathcal{D}_{\neq,1}^{1/3}
	\mathcal{D}_{\neq,4}^{2/3}
	e^{-2\delta_0\nu^{1/2}t},
	\end{align*}
	 which  gives \eqref{w1}.
	
	Similarly, interpolating the weights \(2m+2\) and \(2m+\frac12\) with exponents \(\frac23\) and \(\frac13\), respectively, gives \(2m+\frac32\). Hence,
	\begin{align*}
		\sum_{k\neq 0} |k|^{2m+\frac{3}{2}}\|\omega_k\|_{L^2}^2
		&\le \Bigg( \sum_{k\neq 0} |k|^{2m+2}\|\omega_k\|_{L^2}^2 \Bigg)^{2/3} \Bigg( \sum_{k\neq 0} |k|^{2m+1/2}\|\omega_k\|_{L^2}^2 \Bigg)^{1/3} \\
		&\lesssim \nu^{-5/6} e^{-2\delta_0\nu^{1/2}t} \mathcal{D}_{\neq,1}^{2/3} \mathcal{D}_{\neq,4}^{1/3},
	\end{align*}
	verifying \eqref{w2}.
	
	Finally, to prove the $L^\infty_y$ bound \eqref{w3}, we first apply the one-dimensional Gagliardo--Nirenberg inequality to each mode. Splitting the frequency weights according to the factors appearing in  $\mathcal{E}_{\neq}$, we obtain
	\begin{equation*}
		\sum_{k\neq 0} \|\omega_k\|_{L^\infty}
		\lesssim \sum_{k\neq 0} |k|^{\frac{1}{8}-m} \left( |k|^{2m}\|\omega_k\|_{L^2}^2 \right)^{1/4} \left( |k|^{2m-\frac{1}{2}}\|\partial_y \omega_k\|_{L^2}^2 \right)^{1/4}.
	\end{equation*}
	Applying the discrete H\"older inequality to the sum on the right-hand side yields
	\begin{equation*}
		\sum_{k\neq 0} \|\omega_k\|_{L^\infty} \lesssim \Big( \sum_{k\neq 0} |k|^{\frac{1}{4}-2m} \Big)^{1/2} \Big( \sum_{k\neq 0} |k|^{2m}\|\omega_k\|_{L^2}^2 \Big)^{1/4} \Big( \sum_{k\neq 0} |k|^{2m-\frac{1}{2}}\|\partial_y \omega_k\|_{L^2}^2 \Big)^{1/4}.
	\end{equation*}
	Since $m > \frac{9}{8}$, the power $\frac{1}{4}-2m < -2$, ensuring that the first sum converges to an absolute constant. For the remaining factors, we recall the definition of $E_k[\omega_k]$ in \eqref{Ek_def}, which directly implies $\|\omega_k\|_{L^2}^2 \le 2 E_k$ and $\|\partial_y \omega_k\|_{L^2}^2 \le 2\nu^{-1/2}|k|^{1/2}E_k$. Bounding these sums by the total energy $\mathcal{E}_{\neq}$ extracts the scaling $\nu^{-1/8}$ and the exponential decay factor $e^{-\delta_0\nu^{1/2}t}$, completing the proof of \eqref{w3}.
\end{proof}
\begin{lemma}\label{psi_neq}
	Let $m > {9}/{8}$. The dissipation functionals yield the following $L^\infty_y$ bounds on the fluctuating stream function and its derivatives:
	\begin{align}
		\sum_{k\neq 0} |k|^{2m+2}\|\psi_k\|_{L^\infty}^2
		\lesssim& e^{-2\delta_0 \nu^{1/2}t} \mathcal{D}_{\neq,3},\label{psi_neq1}\\
		\sum_{k\neq 0} |k|^{2m+3}\| \psi_k \|_{L^\infty}^2
		\lesssim& \nu^{-1/3} e^{-2\delta_0\nu^{1/2}t} \mathcal{D}_{\neq,1}^{1/6} \mathcal{D}_{\neq,3}^{1/2} \mathcal{D}_{\neq,4}^{1/3},\label{psi_neq2}\\
		\sum_{\ell\neq 0} |\ell|\,\|\partial_y\psi_\ell\|_{L^\infty}
		\lesssim& \nu^{-1/8} e^{-\delta_0 \nu^{1/2}t} \mathcal{D}_{\neq,3}^{1/4} \mathcal{D}_{\neq,4}^{1/4}.\label{psi_neq3}
	\end{align}
\end{lemma}

\begin{proof}
	For the first estimate \eqref{psi_neq1}, we apply the one-dimensional Gagliardo--Nirenberg interpolation inequality to obtain $\|\psi_k\|_{L^\infty}^2 \lesssim |k|^{-1}\|\nabla_k \psi_k\|_{L^2}^2$. Multiplying by $|k|^{2m+2}$ and summing over the nonzero frequencies yields
	\begin{equation*}
		\sum_{k\neq 0} |k|^{2m+2}\|\psi_k\|_{L^\infty}^2
		\lesssim \sum_{k\neq 0} |k|^{2m+1}\| \nabla_k \psi_k\|^2_{L^2}
		\lesssim e^{-2\delta_0 \nu^{1/2}t} \mathcal{D}_{\neq,3}.
	\end{equation*}
	
	To obtain \eqref{psi_neq2}, we again apply the Gagliardo--Nirenberg inequality, followed by the elliptic regularity estimate $\| k \partial_y \psi_k\|_{L^2} \lesssim \|\omega_k\|_{L^2}$. This allows us to bound the sum by
	\begin{align*}
		\sum_{k\neq 0} |k|^{2m+3}\| \psi_k \|_{L^\infty}^2
		&\lesssim \sum_{k\neq 0} |k|^{2m+1} \| k \psi_k \|_{L^2} \| k \partial_{y} \psi_k\|_{L^2} \\
		&\lesssim \sum_{k\neq 0} \big( |k|^{m+1/2} \| \nabla_k \psi_k \|_{L^2} \big) \big( |k|^{m+1/2} \| \omega_k\|_{L^2} \big).
	\end{align*}
	Applying \eqref{w1} from Lemma~\ref{w_neq} gives \eqref{psi_neq2}.
	
	Finally, for the gradient bound \eqref{psi_neq3}, applying the Gagliardo--Nirenberg inequality and distributing the Fourier weights so that the two resulting factors are controlled by \(\mathcal{D}_{\neq,3}\) and \(\mathcal{D}_{\neq,4}\), respectively, we obtain
	\begin{align*}
		\sum_{\ell\neq 0} |\ell|\,\|\partial_y\psi_\ell\|_{L^\infty}
		&\lesssim \sum_{\ell\neq 0} |\ell|^{\frac{5}{8}-m} \Big( |\ell|^{m+1/2} \|\partial_y\psi_\ell\|_{L^2} \Big)^{1/2} \left( |\ell|^{m+1/4}\|\omega_\ell\|_{L^2} \right)^{1/2} \\
		&\le \Bigg( \sum_{\ell\neq 0} |\ell|^{\frac{5}{4}-2m} \Bigg)^{1/2} \Bigg( \sum_{\ell\neq 0} |\ell|^{2m+1} \|\partial_y\psi_\ell\|_{L^2}^2 \Bigg)^{1/4} \Bigg( \sum_{\ell\neq 0} |\ell|^{2m+1/2} \|\omega_\ell\|_{L^2}^2 \Bigg)^{1/4}.
	\end{align*}
	Since $m > \frac{9}{8}$, the exponent $\frac{5}{4}-2m < -1$, ensuring that the sum $\sum |\ell|^{\frac{5}{4}-2m}$ converges to an absolute constant. The two remaining sums are bounded by \(e^{-2\delta_0\nu^{1/2}t} \mathcal{D}_{\neq,3}\) and \(\nu^{-1/2}e^{-2\delta_0\nu^{1/2}t} \mathcal{D}_{\neq,4}\), respectively. This proves \eqref{psi_neq3}.
\end{proof}
\begin{lemma}\label{mix}
The following weighted elliptic estimate controls $ y \psi_k$ in terms of $ y\omega_k$ and $ \psi_k$. More precisely, for every $k\neq0$, one has
	\begin{equation}\label{fact1}
		|k|^4 \|y\psi_k\|_{L^2}^2 + 2k^2\|y\partial_y\psi_k\|_{L^2}^2
		\le \|y\omega_k\|_{L^2}^2 + 2k^2\|\psi_k\|_{L^2}^2.
	\end{equation}
	Consequently, this implies the summation bound
	\begin{equation}\label{y_psi_infty}
		\sum_{k\neq 0} |k|^{2m+1} \|y\psi_k\|_{L^\infty}^2
		\lesssim e^{-2\delta_0\nu^{1/2}t}\mathcal{D}_{\neq,3} .
	\end{equation}
\end{lemma}

\begin{proof}
	We first establish the modewise estimate \eqref{fact1}. Recalling the elliptic relation $$\omega_k=(\partial_y^2-k^2)\psi_k,$$ and integrating by parts, we obtain
\begin{align*}
\begin{aligned}
\|y\omega_k\|_{L^2}^2
&=
\int_{\mathbb{R}}
y^2
\left|
(\partial_y^2-k^2)\psi_k
\right|^2\,dy
\\
&=
\int_{\mathbb{R}}
y^2|\partial_y^2\psi_k|^2\,dy
-
2\operatorname{Re}
\int_{\mathbb{R}}
y^2k^2\partial_y^2\psi_k\,\overline{\psi_k}\,dy
+
\int_{\mathbb{R}}
y^2|k|^4|\psi_k|^2\,dy \\
&=\| y\partial_{y}^2 \psi_k\|^2_{L^2} + |k|^4\| y  \psi_k\|^2_{L^2} + 2 \|iky\partial_y\psi_k \|^2_{L^2}
-
2\| k \psi_k\|^2_{L^2},
\end{aligned}
\end{align*}
which proves \eqref{fact1}.
	
	We next use \eqref{fact1} to prove the \(L^\infty\) estimate \eqref{y_psi_infty}. By the one-dimensional Gagliardo--Nirenberg inequality and the product rule,
	\begin{equation*}
		\|y\psi_k\|_{L^\infty}^2 \lesssim \|y\psi_k\|_{L^2} \|\partial_y(y\psi_k)\|_{L^2} \le \|y\psi_k\|_{L^2} \big( \|\psi_k\|_{L^2} + \|y\partial_y\psi_k\|_{L^2} \big).
	\end{equation*}
	To preserve the optimal decay rates with respect to the frequency, we apply a weighted Young's inequality to get
	\begin{equation*}
		|k|^{2m+1} \|y\psi_k\|_{L^\infty}^2 \lesssim |k|^{2m+2} \|y\psi_k\|_{L^2}^2 + |k|^{2m} \|\psi_k\|_{L^2}^2 + |k|^{2m} \|y\partial_y\psi_k\|_{L^2}^2.
	\end{equation*}
	On the other hand, dividing \eqref{fact1} by $k^2$ implies that
$$|k|^2 \|y\psi_k\|_{L^2}^2 + \|y\partial_y\psi_k\|_{L^2}^2 \le |k|^{-2}\|y\omega_k\|_{L^2}^2 + 2\|\psi_k\|_{L^2}^2.$$
 Using this to bound the right-hand side of the previous inequality, and noting that $|k| \ge 1$, we find
	\begin{align*}
		\sum_{k\neq 0} |k|^{2m+1} \|y\psi_k\|_{L^\infty}^2
		&\lesssim \sum_{k\neq 0} \left( |k|^{2m-2} \|y\omega_k\|_{L^2}^2 + |k|^{2m} \|\psi_k\|_{L^2}^2 \right) \\
		&\lesssim \sum_{k\neq 0} |k|^{2m+1} \left( \|y\omega_k\|_{L^2}^2 + \|\nabla_k\psi_k\|_{L^2}^2 \right).
	\end{align*}
	Recalling the definition of $\mathcal{D}_{\neq,3}$, the final sum is bounded by $e^{-2\delta_0\nu^{1/2}t}\mathcal{D}_{\neq,3} $, which completes the proof.
\end{proof}

\section{Nonlinear Stability}\label{sec:nonlinear-stability}
\subsection{Zero-mode estimate}
\label{subsec:zero-mode-estimate}

We first derive the evolution equation for the streamwise average.
Since the projection $P_0$ commutes with $\partial_y$ and
$
P_0(\partial_x f)=0
$
for every sufficiently regular function $f$, taking the streamwise
average of \eqref{wholeNS} gives
\begin{align}\label{lingmo}
\partial_t\omega_0-\nu\partial_y^2\omega_0
=
-P_0\bigl(u\cdot\nabla\omega\bigr).
\end{align}
Since the zero mode has no enhanced dissipation, we estimate it through its interaction with the nonzero modes.

\begin{lemma}[Zero-mode energy estimate]\label{lem:zero_mode}
	The nonlinear forcing of the zero mode satisfies the following estimate:
	\begin{equation}\label{zero_mode_energy}
		\frac{1}{2}\frac{\mathrm{d}}{\mathrm{d}t}\mathcal{E}_0 + \mathcal{D}_0
		\lesssim \nu^{-7/12} e^{-2\delta_0\nu^{1/2}t} \mathcal{D}_0^{1/2}\mathcal{D}_{\neq,3}^{1/2}\mathcal{E}_{\neq}^{1/2}.
	\end{equation}
\end{lemma}

\begin{proof}
	Testing the zero-mode equation of \eqref{lingmo} against $\omega_0$ in $L^2$, and its $y$-derivative against $\alpha\nu^{1/3}\partial_y \omega_0$, we obtain the exact energy balance
	\begin{equation*}
		\frac{1}{2}\frac{\mathrm{d}}{\mathrm{d}t}\mathcal{E}_0 + \mathcal{D}_0 = \mathcal{N}_0,
	\end{equation*}
	where the nonlinear convective contribution is given by
	\begin{align*}
		\mathcal{N}_0
		&= -\operatorname{Re}\left\langle \omega_0, (u\cdot\nabla \omega)_0 \right\rangle
		- \alpha\nu^{1/3} \operatorname{Re} \left\langle \partial_y \omega_0, \partial_y (u\cdot\nabla \omega)_0 \right\rangle \\
		&=: \mathcal{N}_0^{(1)} + \mathcal{N}_0^{(2)}.
	\end{align*}

	For the first contribution, we expand the convective term and integrate by parts to shift the derivative onto $\omega_0$:
	\begin{equation*}
		\mathcal{N}_0^{(1)} = \operatorname{Re} \int_{\mathbb{R}} \omega_0\, \partial_y \bigg( \sum_{\ell\neq 0} \psi_\ell\, i\ell \omega_{-\ell} \bigg) \,\mathrm{d}y
		= -\operatorname{Re} \sum_{\ell\neq 0} \int_{\mathbb{R}} \partial_y \omega_0\, (i\ell \psi_\ell) \omega_{-\ell} \,\mathrm{d}y.
	\end{equation*}
	Applying the Cauchy--Schwarz inequality in $y$ and extracting the supremum of the velocity fluctuation gives
	\begin{equation*}
		\big|\mathcal{N}_0^{(1)}\big| \le \|\partial_y \omega_0\|_{L^2} \sum_{\ell\neq 0} \|\ell\psi_\ell\|_{L^\infty} \|\omega_{-\ell}\|_{L^2}.
	\end{equation*}
	Recalling that $\|\partial_y \omega_0\|_{L^2} \le \nu^{-1/2}\mathcal{D}_0^{1/2}$,  Lemma~\ref{w_neq} and \ref{psi_neq}  give
	\begin{equation}\label{NL0_1_bound}
		\big|\mathcal{N}_0^{(1)}\big| \lesssim \nu^{-1/2} e^{-2\delta_0 \nu^{1/2}t} \mathcal{D}_0^{1/2} \mathcal{D}_{\neq,3}^{1/2} \mathcal{E}_{\neq}^{1/2}.
	\end{equation}

	For the higher-order interaction $\mathcal{N}_0^{(2)}$, a single integration by parts yields
	\begin{equation*}
		\mathcal{N}_0^{(2)} = -\alpha\nu^{1/3} \operatorname{Re} \sum_{\ell\neq 0} \int_{\mathbb{R}} \partial_y^2 \omega_0 \Big( i\ell\partial_y\psi_\ell\,\omega_{-\ell} + i\ell\psi_\ell\,\partial_y \omega_{-\ell} \Big) \,\mathrm{d}y.
	\end{equation*}
	Applying the Gagliardo--Nirenberg interpolation inequality to the $L^\infty$ norms of the velocity fluctuations, we can bound this term by
	\begin{equation*}
		\big|\mathcal{N}_0^{(2)}\big| \lesssim \nu^{1/3} \|\partial_y^2\omega_0\|_{L^2} \sum_{\ell\neq 0} |\ell| \left( \|\omega_\ell\|_{L^2}^{1/2}\|\partial_y\psi_\ell\|_{L^2}^{1/2}\|\omega_{-\ell}\|_{L^2} + \|\partial_y\psi_\ell\|_{L^2}^{1/2}\|\psi_\ell\|_{L^2}^{1/2}\|\partial_y \omega_{-\ell}\|_{L^2} \right).
	\end{equation*}
	Note that $\nu^{1/3}\|\partial_y^2\omega_0\|_{L^2} \lesssim \nu^{-1/3}\mathcal{D}_0^{1/2}$, applying the weighted H\"older inequality with the frequency weights appearing in \(\mathcal{E}_{\neq}\)  and \(\mathcal{D}_{\neq,3}\), we obtain
	\begin{equation*}
		\big|\mathcal{N}_0^{(2)}\big| \lesssim \nu^{-1/3}\mathcal{D}_0^{1/2} \times \nu^{-1/4} e^{-2\delta_0\nu^{1/2}t} \mathcal{D}_{\neq,3}^{1/2}\mathcal{E}_{\neq}^{1/2}
		= \nu^{-7/12} e^{-2\delta_0\nu^{1/2}t} \mathcal{D}_0^{1/2} \mathcal{D}_{\neq,3}^{1/2}\mathcal{E}_{\neq}^{1/2}.
	\end{equation*}

	Since \(0<\nu<1\) and \(7/12>1/2\), we have $ \nu^{-7/12} \ge \nu^{-1/2}$. Summing the two nonlinear contributions yields the desired estimate \eqref{zero_mode_energy}.
\end{proof}

\subsection{Nonzero-mode estimate}
We next estimate the nonzero modes. We use the modewise hypocoercive energy \(E_k\) together with the exponential time weight in \eqref{def_Eneq}.

For each $k \neq 0$, the evolution of the perturbed vorticity is governed by
\begin{equation*}
	\partial_t \omega_k -\mathcal{L}_k + \mathcal{N}_k=0,
\end{equation*}
where the linear generator (capturing transport, nonlocal coupling, and diffusion) is
\begin{equation*}
	\mathcal{L}_k = -iky^2\omega_k + 2ik\psi_k + \nu\Delta_k \omega_k,
\end{equation*}
and the nonlinear convective forcing takes the convolution form
\begin{equation*}
	\mathcal{N}_k =(u\cdot\nabla \omega)_k = -\sum_{\ell\in\mathbb{Z}} \partial_y\psi_\ell\, i(k-\ell)\omega_{k-\ell} + \sum_{\ell\in\mathbb{Z}} i\ell\psi_\ell\,\partial_y \omega_{k-\ell}.
\end{equation*}

Recalling the definition of the exponentially weighted energy $\mathcal{E}_{\neq}$ in \eqref{def_Eneq}, its time derivative can be  expressed by
\begin{equation}\label{E_neq}
	\frac{\mathrm{d}}{\mathrm{d}t}\mathcal{E}_{\neq} = \mathcal{L}_{\neq} + \mathcal{N}_{\neq} + 2\delta_0\nu^{1/2}\mathcal{E}_{\neq},
\end{equation}
where the total linear and nonlinear contributions are given respectively by
\begin{align*}
\mathcal{L}_{\neq}
&=
2e^{2\delta_0\nu^{1/2}t}
\sum_{k\neq 0}
|k|^{2m}
\operatorname{Re}
\left\langle
\omega_k,\mathcal{L}_k
\right\rangle
+
2\alpha\nu^{1/2}
e^{2\delta_0\nu^{1/2}t}
\sum_{k\neq 0}
|k|^{2m-\frac12}
\operatorname{Re}
\left\langle
\nabla_k \omega_k,\nabla_k \mathcal{L}_k
\right\rangle
\nn\\
&\quad
+
\beta e^{2\delta_0\nu^{1/2}t}
\sum_{k\neq 0}
|k|^{2m-1}
\left(
\operatorname{Re}
\left\langle
\partial_y \mathcal{L}_k,iky\omega_k
\right\rangle
+
\operatorname{Re}
\left\langle
iky\mathcal{L}_k,\partial_y \omega_k
\right\rangle
\right)
\nn\\
&\quad
+
2\gamma\nu^{-1/2}
e^{2\delta_0\nu^{1/2}t}
\sum_{k\neq 0}
|k|^{2m-\frac32}
(\operatorname{Re}
\left\langle
iky\omega_k,iky\mathcal{L}_k
\right\rangle-\operatorname{Re}
\left\langle
ik \mathcal{L}_k,ik\psi_k
\right\rangle),
\end{align*}
and
\begin{align*}
\begin{aligned}
\mathcal{N}_{\neq}
&=
2e^{2\delta_0\nu^{1/2}t}
\sum_{k\neq 0}
|k|^{2m}
\operatorname{Re}
\left\langle
\omega_k,\mathcal{N}_k
\right\rangle
+
2\alpha\nu^{1/2}
e^{2\delta_0\nu^{1/2}t}
\sum_{k\neq 0}
|k|^{2m-\frac12}
\operatorname{Re}
\left\langle
\nabla_k \omega_k,\nabla_k \mathcal{N}_k
\right\rangle
\\
&\quad
+
\beta e^{2\delta_0\nu^{1/2}t}
\sum_{k\neq 0}
|k|^{2m-1}
\left(
\operatorname{Re}
\left\langle
\partial_y \mathcal{N}_k,iky\omega_k
\right\rangle
+
\operatorname{Re}
\left\langle
iky\mathcal{N}_k,\partial_y \omega_k
\right\rangle
\right)
\\
&\quad
+
2\gamma\nu^{-1/2}
e^{2\delta_0\nu^{1/2}t}
\sum_{k\neq 0}
|k|^{2m-\frac32}
\operatorname{Re}
\left\langle
iky\omega_k,iky\mathcal{N}_k
\right\rangle\nn\\
&\quad-
2\gamma\nu^{-1/2}
e^{2\delta_0\nu^{1/2}t}
\sum_{k\neq 0}
|k|^{2m-\frac32}
\operatorname{Re}
\left\langle
ik\mathcal{N}_k,ik\psi_k
\right\rangle
\\
&=:
\sum_{j=1}^5 \mathcal{N}_{\neq}^{(j)}.
\end{aligned}
\end{align*}

By the equivalence \eqref{Ekdengjia} and the definitions of $ \mathcal{D}_{\neq}$, we have
\begin{align*}
&\quad\nu^{1/2}
\sum_{k\neq 0}
|k|^{2m}|k|^{1/2}
E_k[\omega_k]\\
&\lesssim \nu^{1/2}
\sum_{k\neq 0}
|k|^{2m}|k|^{1/2} \left(\|\omega_k\|_{L^2}^2 + \nu^{1/2} |k|^{-1/2} \|\nabla_k \omega_k\|_{L^2}^2 + \nu^{-1/2} |k|^{1/2} \left[\|y\omega_k\|_{L^2}^2
+
2\|\nabla_k\psi_k\|_{L^2}^2 \right]\right)\\
&\lesssim e^{-2\delta_0\nu^{1/2}t}\mathcal{D}_{\neq}.
\end{align*}
Combining the preceding estimate with the linear coercivity estimate \eqref{Ek_dissipation_coercive}, and choosing \(\delta_0>0\) sufficiently small, we obtain
\begin{align*}
\mathcal{L}_{\neq}
\le
-8\delta_0 \mathcal{D}_{\neq}
-
8\delta_0\nu^{1/2}
e^{2\delta_0\nu^{1/2}t}
\sum_{k\neq 0}
|k|^{2m+\frac12}
E_k[\omega_k].
\end{align*}
Substituting the preceding linear estimate into \eqref{E_neq}, it remains to estimate $\mathcal{N}_{\neq}^{(1)},\ldots,\mathcal{N}_{\neq}^{(5)}$. Their proofs are deferred to Section \ref{sec:proof-nonlinear-lemmas}.

\medskip
We begin with the estimate of $\mathcal{N}_{\neq}^{(1)}$.
\begin{lemma}[Estimate of $\mathcal{N}_{\neq}^{(1)}$]\label{le_N1}
	The contribution of \(\mathcal{N}_k\) to the \(L^2\)-part of \(E_k\) satisfies
	\begin{equation}\label{N1}
		\bigl|\mathcal{N}_{\neq}^{(1)}\bigr| \lesssim \nu^{-2/3} \mathcal{E}^{1/2} \Big(
		\mathcal{D}_{\neq,4}^{1/2} \mathcal{D}_{\neq,3}^{1/2} +
		\mathcal{D}_{\neq,1}^{1/2} \mathcal{D}_{\neq,3}^{1/2} +
		\mathcal{D}_{\neq,1}^{1/3} \mathcal{D}_{\neq,4}^{2/3} +
		\mathcal{D}_{\neq,1}^{1/4} \mathcal{D}_{\neq,3}^{1/4} \mathcal{D}_{\neq,4}^{1/2} \Big).
	\end{equation}
\end{lemma}

We next estimate $\mathcal{N}_{\neq}^{(2)}$.
\begin{lemma}[Estimate of $\mathcal{N}_{\neq}^{(2)}$]\label{le_N2}
The nonlinear term associated with the $\nabla_k\omega_k$ component of
the hypocoercive energy satisfies
\begin{equation}\label{N2}
\begin{aligned}
\bigl|\mathcal{N}_{\neq}^{(2)}\bigr|
\lesssim \nu^{-2/3}\mathcal{E}^{1/2}\Bigl(
&\mathcal{D}_{\neq,1}^{2/3}\mathcal{D}_{\neq,4}^{1/3}
+\mathcal{D}_{\neq,2}^{1/2}\mathcal{D}_{\neq,3}^{1/2}
+\mathcal{D}_{\neq,2}^{1/2}\mathcal{D}_{\neq,4}^{1/2}
+\mathcal{D}_{\neq,2}^{1/2}
 \mathcal{D}_{\neq,1}^{1/3}
 \mathcal{D}_{\neq,4}^{1/6}
\Bigr).
\end{aligned}
\end{equation}
\end{lemma}

The mixed term requires a time-integrated estimate of the zero-mode velocity. Its bound is stated separately in the following lemma.
\begin{lemma}[Estimate of $\mathcal{N}_{\neq}^{(3)}$]\label{le_N3}
The nonlinear contribution associated with the mixed spatial weights satisfies
	\begin{align}\label{N3}
		\bigl|\mathcal{N}_{\neq}^{(3)}\bigr|   &\lesssim \nu^{-2/3}
		\mathcal{E}^{1/2}
		\mathcal{D}_{\neq} + 2\delta_0
		\left(
		\mathcal{D}_{\neq,1}+\mathcal{D}_{\neq,4}
		\right)
		\nonumber\\
		&\quad +\nu^{-2/3-1/4}
		\mathcal{E}^{1/2}
		\mathcal{D}_{\neq}
		\left(
		\int_0^t
		\left(
		\sum_{k\neq 0}
		|k|^{2m}\|\nabla_k\psi_k\|_{L^2}^2
		\right)^{1/2}
		\left(
		\sum_{k\neq 0}
		|k|^{2m}\|\omega_k\|_{L^2}^2
		\right)^{1/2}
		d\tau
		\right)^{1/4}.
	\end{align}

\end{lemma}

The \(y\omega_k\)-term is estimated using the weighted elliptic estimate in Lemma~\ref{mix} and the frequency decomposition \eqref{frequency_partition}.

\begin{lemma}[Estimate of $\mathcal{N}_{\neq}^{(4)}$]\label{le_N4}
	\begin{align}\label{N4}
		\bigl|\mathcal{N}_{\neq}^{(4)}\bigr| \lesssim \nu^{-2/3}
		\mathcal{E}^{1/2}
		\left(\mathcal{D}_{\neq,3} +
		\mathcal{D}_{\neq,5}^{1/2} \mathcal{D}^{1/4}_{\neq,3} \mathcal{D}^{1/12}_{\neq,1} \mathcal{D}^{1/6}_{\neq,4} +
		\mathcal{D}_{\neq,5}^{1/2}
		\mathcal{D}_{\neq,3}^{1/2} + \mathcal{D}_{\neq,3}^{1/4}
		\mathcal{D}_{\neq,4}^{1/4}
		\mathcal{D}_{\neq,5}^{1/2}\right).
	\end{align}
\end{lemma}

For \(\mathcal{N}_{\neq}^{(5)}\), we use \(\omega_k=\Delta_k\psi_k\), integrate by parts in \(y\),  and apply the decomposition \eqref{nonlinear_interaction_decomposition}.

\begin{lemma}[Estimate of $\mathcal{N}_{\neq}^{(5)}$]\label{le_N5}
	\begin{align}\label{N5}
	\bigl|\mathcal{N}_{\neq}^{(5)}\bigr|
	\lesssim
	\nu^{-5/8}
	\mathcal{E}_{\neq}^{1/2}
	\mathcal{D}_{\neq,3}^{1/2}
	\mathcal{D}_{\neq,3}^{1/4}
	\mathcal{D}_{\neq,4}^{1/4}
	+
	\nu^{-1/2}
	\mathcal{E}^{1/2}
	\mathcal{D}_{\neq,3}.
	\end{align}
	
\end{lemma}

\medskip
\subsection{Completion of the proof of Theorem~\ref{thm1}}
With  the above five lemmas  in hand, we  begin to  use the bootstrap argument to complete the  proof of Theorem \ref{thm1}. By Lemma~\ref{lem:zero_mode} and Lemmas~\ref{le_N1}--\ref{le_N5},  we obtain
\begin{align}\label{E0}
\frac12\frac{\mathrm{d}}{\mathrm{d}t}\mathcal{E}_0
+
\mathcal{D}_0
\lesssim
\nu^{-7/12}
\mathcal{D}_0^{1/2}
\mathcal{D}_{\neq,3}^{1/2}
\mathcal{E}_{\neq}^{1/2},
\end{align}
and
\begin{align*}
\begin{aligned}
\frac12\frac{\mathrm{d}}{\mathrm{d}t}\mathcal{E}_{\neq}
&+
8\delta_0\mathcal{D}_{\neq}
+
8\delta_0\nu^{1/2}
e^{2\delta_0\nu^{1/2}t}
\sum_{k\neq 0}
|k|^{2m+\frac12}E_k[\omega_k]
\\
&\le
C\nu^{-2/3}\mathcal{E}^{1/2}\left(\mathcal{D}_{\neq} + \mathcal{D}_0\right) + 2\delta_0
\left(
\mathcal{D}_{\neq,1}+\mathcal{D}_{\neq,4}
\right)
\nonumber\\
&\quad +C\nu^{-2/3-1/4}
\mathcal{E}^{1/2}
\mathcal{D}_{\neq}
\left(
\int_0^t
\left(
\sum_{k\neq 0}
|k|^{2m}\|\nabla_k\psi_k\|_{L^2}^2
\right)^{1/2}
\left(
\sum_{k\neq 0}
|k|^{2m}\|\omega_k\|_{L^2}^2
\right)^{1/2}
d\tau
\right)^{1/4}.\nonumber
\end{aligned}
\end{align*}

For $t\ge0$, define
\begin{align*}
\mathcal{E}_{\mathrm{tot}}(t)
:=
\sup_{0\le s\le t}\mathcal{E}(s)
+\int_0^t\bigl(\mathcal{D}_{\neq}(s)+\mathcal{D}_0(s)\bigr)\,\mathrm{d}s.
\end{align*}
The only term that is not immediately bounded by $\mathcal{E}_{\mathrm{tot}}$ is
\begin{align*}
\mathcal{I}(t)
:=\int_0^t
\left(\sum_{k\neq0}|k|^{2m}\|\nabla_k\psi_k\|_{L^2}^2\right)^{1/2}
\left(\sum_{k\neq0}|k|^{2m}\|\omega_k\|_{L^2}^2\right)^{1/2}\,\mathrm{d}\tau.
\end{align*}
By the definitions of $\mathcal{D}_{\neq,4}$ and $\mathcal{D}_{\neq,3}$, and since $|k|\ge1$ for $k\neq0$,
\begin{align*}
\mathcal{I}(t)
&\lesssim\nu^{-1/4}
\int_0^t e^{-2\delta_0\nu^{1/2}\tau}
\mathcal{D}_{\neq,3}^{1/2}(\tau)\mathcal{D}_{\neq,4}^{1/2}(\tau) \,\mathrm{d}\tau
\\
&\lesssim\nu^{-1/4}\left(\int_0^t
\mathcal{D}_{\neq,3}(\tau)\,\mathrm{d}\tau\right)^{1/2}\left(\int_0^t
\mathcal{D}_{\neq,4}(\tau)\,\mathrm{d}\tau\right)^{1/2}
\\
&\lesssim\nu^{-1/4}\mathcal{E}_{\mathrm{tot}}(t),
\end{align*}
where the last step follows from Cauchy--Schwarz in time. Hence
\begin{align}\label{I_bound}
\mathcal{I}(t)^{1/4}
\lesssim \nu^{-1/16}\mathcal{E}_{\mathrm{tot}}(t)^{1/4}.
\end{align}
After integrating the combined energy inequalities, absorbing the $2\delta_0(\mathcal{D}_{\neq,1}+\mathcal{D}_{\neq,4})$ contribution into the left-hand side, and using \eqref{I_bound}, we obtain
\begin{align}\label{bootstrap_ineq}
\mathcal{E}_{\mathrm{tot}}(t)
\le C_0\mathcal{E}_{\mathrm{tot}}(0)
+C_1\nu^{-2/3}\mathcal{E}_{\mathrm{tot}}(t)^{3/2}
+C_1\nu^{-47/48}\mathcal{E}_{\mathrm{tot}}(t)^{7/4}.
\end{align}
The initial assumption \eqref{initial_data} and the equivalence of the modewise energy imply
\begin{align*}
\mathcal{E}_{\mathrm{tot}}(0)\le C_2\varepsilon^2\nu^{4/3}.
\end{align*}
Fix $M=4C_0C_2$ and suppose, as a bootstrap hypothesis, that
\begin{align*}
\mathcal{E}_{\mathrm{tot}}(t)\le M\varepsilon^2\nu^{4/3}
\end{align*}
on a maximal interval $[0,T]$. Substituting this bound into the two nonlinear terms in \eqref{bootstrap_ineq} gives
\begin{align*}
\nu^{-2/3}\mathcal{E}_{\mathrm{tot}}^{3/2}
&\le M^{3/2}\varepsilon^3\nu^{4/3},\\
\nu^{-47/48}\mathcal{E}_{\mathrm{tot}}^{7/4}
&\le M^{7/4}\varepsilon^{7/2}\nu^{65/48}
\le M^{7/4}\varepsilon^{7/2}\nu^{4/3},
\end{align*}
where we used $0<\nu<1$ and $65/48>4/3$. Therefore, if $\varepsilon_0$ is chosen sufficiently small depending only on $C_0,C_1,C_2$, then \eqref{bootstrap_ineq} improves the bootstrap bound to
\begin{align*}
\mathcal{E}_{\mathrm{tot}}(t)\le \frac{M}{2}\varepsilon^2\nu^{4/3},
\qquad 0\le t\le T.
\end{align*}
The standard continuity argument then yields $T=+\infty$ and
\begin{align*}
\mathcal{E}_{\mathrm{tot}}(t)\lesssim \varepsilon^2\nu^{4/3},
\qquad t\ge0.
\end{align*}
Finally, the equivalence defining $E_k[\omega_k]$ and the exponential factor in $\mathcal{E}_{\neq}$ give \eqref{nonzero}, while the definition of $\mathcal{E}_0$ gives \eqref{zero}. This completes the proof of Theorem~\ref{thm1}.\hfill $\square$

\medskip
\medskip
\section{Proof of  Lemmas \ref{le_N1}--\ref{le_N5}}\label{sec:proof-nonlinear-lemmas}
It remains to prove Lemmas \ref{le_N1}--\ref{le_N5}. For clarity, we
first specify the notation used to distinguish the different
interactions in the Fourier convolution. Throughout this section, the
first subscript refers to the Fourier mode of the velocity factor, while
the second subscript refers to the Fourier mode of the vorticity factor.
More precisely, for a fixed output mode $k\neq0$, the convolution is
decomposed according to
\begin{equation*}
\sum_{\ell\in\mathbb Z} V_\ell W_{k-\ell}
=
V_0W_k
+
V_kW_0
+
\sum_{\substack{\ell\neq0,\  k-\ell\neq0}}
V_\ell W_{k-\ell},
\end{equation*}
where $V_\ell$ denotes a velocity factor and $W_{k-\ell}$ denotes a
vorticity factor. Accordingly, we use the following notation:
\begin{itemize}
  \item $\mathcal{N}_{0,\neq}$ denotes the mean--fluctuation
  interaction, corresponding to $\ell=0$: the velocity factor is the
  zero mode and the vorticity factor is the nonzero mode;

\medskip
  \item $\mathcal{N}_{\neq,0}$ denotes the fluctuation--mean
  interaction, corresponding to $\ell=k$: the velocity factor is the
  nonzero mode and the vorticity factor is the zero mode;

\medskip
  \item $\mathcal{N}_{\neq,\neq}$ denotes the fluctuation--fluctuation
  interaction, corresponding to
  $\ell\neq0$ and $k-\ell\neq0$: both the velocity and vorticity factors
  are nonzero modes.
\end{itemize}

When necessary, we use the superscript $(j)$ to indicate the
contribution associated with $\mathcal{N}_{\neq}^{(j)}$. Thus, for
$j=1,\ldots,5$, we write
\begin{equation}\label{nonlinear_interaction_decomposition}
\mathcal{N}_{\neq}^{(j)}
=
\mathcal{N}_{0,\neq}^{(j)}
+
\mathcal{N}_{\neq,0}^{(j)}
+
\mathcal{N}_{\neq,\neq}^{(j)}.
\end{equation}
In particular, the labels in \eqref{nonlinear_interaction_decomposition}
refer to the modes of the two factors in the nonlinear convolution and
not to the output mode, which is always restricted to $k\neq0$ in the
definition of $\mathcal{N}_{\neq}^{(j)}$.

We shall repeatedly use the following frequency decomposition. For
$\ell\neq0$ and $k-\ell\neq0$, define
\begin{align*}
\Omega_{\mathrm{LH}}
&:=
\left\{
(k,\ell)\in\mathbb Z^2:
\ell\neq0,\ k-\ell\neq0,\
|k-\ell|<\frac{|\ell|}{2}
\right\},
\nonumber\\
\Omega_{\mathrm{HL}}
&:=
\left\{
(k,\ell)\in\mathbb Z^2:
\ell\neq0,\ k-\ell\neq0,\
|k-\ell|\ge\frac{|\ell|}{2}
\right\}.
\end{align*}
Then
\begin{equation}
\label{frequency_partition}
\Big\{
(k,\ell)\in\mathbb Z^2:
\ell\neq0,\ k-\ell\neq0
\Big\}
=
\Omega_{\mathrm{LH}}\mathbin{\dot\cup}\Omega_{\mathrm{HL}},
\end{equation}
where $\dot\cup$ denotes the disjoint union. The labels
$\mathrm{LH}$ and $\mathrm{HL}$ indicate, respectively, the region in
which $|k-\ell|$ is smaller than $|\ell|$ and the complementary region
in which $|k-\ell|$ is comparable to or larger than $|\ell|$.

On $\Omega_{\mathrm{LH}}$, we have
\begin{align*}
|k-\ell|<\frac{|\ell|}{2},
\qquad
\frac{|\ell|}{2}<|k|
\le |\ell|+|k-\ell|
<\frac32|\ell|,
\end{align*}
and hence
\begin{equation}
\label{frequency_comparability_LH}
|k|\simeq|\ell|
\qquad\text{on }\Omega_{\mathrm{LH}}.
\end{equation}
On $\Omega_{\mathrm{HL}}$, the triangle inequality gives
\begin{align*}
|\ell|
\le |k|+|k-\ell|,
\end{align*}
while the defining condition implies
$|\ell|\le2|k-\ell|$. Therefore,
\begin{equation}
\label{frequency_comparability_HL}
|k|\lesssim|k-\ell|,
\qquad
|\ell|\lesssim|k-\ell|
\qquad\text{on }\Omega_{\mathrm{HL}}.
\end{equation}
Consequently, whenever a sum is taken over the fluctuation--fluctuation
interactions, we shall use the decomposition
\begin{align*}
\sum_{\substack{\ell\neq0,\  k-\ell\neq0}}
F_{k,\ell}
&=
\sum_{(k,\ell)\in\Omega_{\mathrm{LH}}}F_{k,\ell}
+
\sum_{(k,\ell)\in\Omega_{\mathrm{HL}}}F_{k,\ell}.
\end{align*}
The same notation and decomposition will be used throughout the proofs
of Lemmas \ref{le_N1}--\ref{le_N5}.

\medskip
We begin with the proof of Lemma \ref{le_N1}, namely, the estimate for
the contribution $\mathcal{N}_{\neq}^{(1)}$.
\subsection{Proof of Lemma \ref{le_N1}}
\begin{proof}
By the definition of $\mathcal{N}^{(j)}$ and the first component of the
modewise hypocoercive energy, we have
	\begin{align*}
		\mathcal{N}_{\neq}^{(1)}
		&= -2e^{2\delta_0\nu^{1/2}t} \sum_{k\neq 0} |k|^{2m} \operatorname{Re} \Big\langle \omega_k, \sum_{\ell\in\mathbb{Z}} \partial_y\psi_\ell\, i(k-\ell)\omega_{k-\ell} \Big\rangle \\
		&\quad + 2e^{2\delta_0\nu^{1/2}t} \sum_{k\neq 0} |k|^{2m} \operatorname{Re} \Big\langle \omega_k, \sum_{\ell\in\mathbb{Z}} i\ell\psi_\ell\,\partial_y \omega_{k-\ell} \Big\rangle.
	\end{align*}
	In accordance with the convention introduced at the beginning of this
section, we decompose
\begin{align*}
\mathcal{N}_{\neq}^{(1)}
=
\mathcal{N}_{0,\neq}^{(1)}
+
\mathcal{N}_{\neq,0}^{(1)}
+
\mathcal{N}_{\neq,\neq}^{(1)}.
\end{align*}

	\smallskip\noindent\textbf{Mean--fluctuation interaction ($\ell=0$).}

	The zero-mode stream function $\psi_0$ depends  on $y$. Because $\partial_y\psi_0$ is real-valued, the integrand is purely imaginary, which immediately forces its real part to vanish identically:
	\begin{equation*}
		\mathcal{N}_{0,\neq}^{(1)} = -2e^{2\delta_0\nu^{1/2}t} \sum_{k\neq 0} |k|^{2m} \operatorname{Re} \left\langle \omega_k, ik (\partial_y\psi_0) \omega_k \right\rangle = 0.
	\end{equation*}

	\smallskip\noindent\textbf{Fluctuation--mean interaction ($k=\ell$).}

	When the fluctuation acts on the mean flow, the relevant term reduces to
	\begin{equation*}
		\mathcal{N}_{\neq,0}^{(1)} = 2e^{2\delta_0\nu^{1/2}t} \sum_{k\neq 0} |k|^{2m} \operatorname{Re} \left\langle \omega_k, ik\psi_k\,\partial_y \omega_0 \right\rangle.
	\end{equation*}
	Applying the Cauchy--Schwarz inequality in $y$ and recalling that $\|\partial_y \omega_0\|_{L^2} \le \alpha^{-1/2} \nu^{-1/6} \mathcal{E}_0^{1/2}$, we obtain
	\begin{equation*}
		\mathcal{N}_{\neq,0}^{(1)} \lesssim \nu^{-1/6} \mathcal{E}_0^{1/2} e^{2\delta_0\nu^{1/2}t} \sum_{k\neq 0} |k|^{2m+1} \|\omega_k\|_{L^2} \|\psi_k\|_{L^\infty}.
	\end{equation*}
	Moreover, we invoke the discrete Cauchy--Schwarz inequality alongside the $L^\infty$ bound \eqref{psi_neq1} from Lemma~\ref{psi_neq}:
	\begin{align*}
		\sum_{k\neq 0} |k|^{2m+1} \|\omega_k\|_{L^2} \|\psi_k\|_{L^\infty}
		&\le \Big( \sum_{k\neq 0} |k|^{2m} \|\omega_k\|_{L^2}^2 \Big)^{1/2} \Big( \sum_{k\neq 0} |k|^{2m+2} \|\psi_k\|_{L^\infty}^2 \Big)^{1/2} \\
		&\lesssim \big( \nu^{-1/2} e^{-2\delta_0\nu^{1/2}t} \mathcal{D}_{\neq,4} \big)^{1/2} \big( e^{-2\delta_0\nu^{1/2}t} \mathcal{D}_{\neq,3} \big)^{1/2}.
	\end{align*}
	Multiplying by the prefactors yields $\mathcal{N}_{\neq,0}^{(1)} \lesssim \nu^{-5/12} \mathcal{E}_0^{1/2} \mathcal{D}_{\neq,4}^{1/2} \mathcal{D}_{\neq,3}^{1/2}$, which is bounded by the first term in \eqref{N1} since $\nu \in (0,1)$.

\smallskip\noindent\textbf{Fluctuation--fluctuation interaction ($k\neq\ell, \ell\neq 0$).}

	For the fully fluctuating interactions, the algebraic commutator identity $-i(k-\ell)\partial_y\psi_\ell\,\omega_{k-\ell} + i\ell\psi_\ell\,\partial_y\omega_{k-\ell} = i\ell\partial_y(\psi_\ell \omega_{k-\ell}) - ik\partial_y\psi_\ell \omega_{k-\ell}$ permits us to decompose the convolution as $\mathcal{N}_{\neq,\neq}^{(1)} = \mathcal{T}_1 + \mathcal{T}_2$, where
	\begin{align*}
		\mathcal{T}_1 &= 2e^{2\delta_0\nu^{1/2}t} \sum_{k\neq 0} |k|^{2m} \operatorname{Re} \Big\langle \omega_k, \sum_{\ell\neq 0, \ell\neq k} i\ell \partial_y (\psi_\ell \omega_{k-\ell}) \Big\rangle, \\
		\mathcal{T}_2 &= -2e^{2\delta_0\nu^{1/2}t} \sum_{k\neq 0} |k|^{2m} \operatorname{Re} \Big\langle \omega_k, \sum_{\ell\neq 0, \ell\neq k} ik \partial_y \psi_\ell\, \omega_{k-\ell} \Big\rangle.
	\end{align*}
	
	In $\mathcal{T}_1$, transferring the derivative onto $\omega_k$ via integration by parts and applying the weight inequality $|k|^{2m} \lesssim |\ell|^{2m} + |k-\ell|^{2m}$ alongside Young's inequality for discrete convolutions yields
	\begin{align*}
		|\mathcal{T}_1| &\lesssim e^{2\delta_0\nu^{1/2}t} \sum_{k, \ell} |k|^{2m} \|\partial_y \omega_k\|_{L^2} \|\ell\psi_\ell\|_{L^\infty} \|\omega_{k-\ell}\|_{L^2} \lesssim \nu^{-1/2} \mathcal{D}_{\neq,1}^{1/2} \mathcal{D}_{\neq,3}^{1/2} \mathcal{E}_{\neq}^{1/2}.
	\end{align*}
	We next estimate $\mathcal{T}_2$. By the frequency partition
\eqref{frequency_partition}, we write
\begin{align*}
\mathcal{T}_2
&=
\left.\mathcal{T}_2\right|_{\Omega_{\mathrm{LH}}}
+
\left.\mathcal{T}_2\right|_{\Omega_{\mathrm{HL}}}.
\end{align*}

On $\Omega_{\mathrm{HL}}$, we have
$
|\ell|\lesssim |k-\ell|$, $
|k|\lesssim |k-\ell|,
$
by \eqref{frequency_comparability_HL}. Hence,
\begin{align*}
|k|^{2m+1}
\lesssim
|k|^{m+\frac12}|k-\ell|^{m+\frac12}.
\end{align*}
Therefore, by the Cauchy--Schwarz inequality and the weighted discrete
convolution estimate,
\begin{align*}
\left|
\left.\mathcal{T}_2\right|_{\Omega_{\mathrm{HL}}}
\right|
\lesssim&
e^{2\delta_0\nu^{1/2}t}
\sum_{(k,\ell)\in\Omega_{\mathrm{HL}}}
\Bigl(
|k|^{m+\frac12}\|\omega_k\|_{L_y^2}
\Bigr)
\Bigl(
|k-\ell|^{m+\frac12}
\|\omega_{k-\ell}\|_{L_y^2}
\Bigr)
\|\partial_y\psi_\ell\|_{L_y^\infty}
\nonumber\\
\lesssim&
e^{2\delta_0\nu^{1/2}t}
\Big(
\sum_{q\neq0}
|q|^{2m+1}\|\omega_q\|_{L_y^2}^2
\Big)
\Big(
\sum_{\ell\neq0}
\|\partial_y\psi_\ell\|_{L_y^\infty}
\Big).
\end{align*}

We next estimate the last factor. By the one-dimensional
Gagliardo--Nirenberg inequality applied to $\partial_y\psi_\ell$, we
have
\begin{align}\label{nonzero_mode_velocity_sum}
\sum_{\ell\neq0}
\|\partial_y\psi_\ell\|_{L_y^\infty}
&\lesssim
\sum_{\ell\neq0}
|\ell|^{-1/2}\|\omega_\ell\|_{L_y^2}
\le
\Big(
\sum_{\ell\neq0}|\ell|^{-2m-1}
\Big)^{1/2}
\Big(
\sum_{\ell\neq0}
|\ell|^{2m}\|\omega_\ell\|_{L_y^2}^2
\Big)^{1/2}.
\end{align}
The first series on the right-hand side is finite for every $m>0$.
Furthermore, by the coercivity of $E_\ell$ and the definition of
$\mathcal E_{\neq}$,
\begin{equation}\label{nonzero_mode_vorticity_energy}
\sum_{\ell\neq0}
|\ell|^{2m}\|\omega_\ell(t)\|_{L_y^2}^2
\lesssim
e^{-2\delta_0\nu^{1/2}t}\mathcal E_{\neq}(t).
\end{equation}
Consequently,
\begin{equation}\label{nonzero_mode_sobolev_estimate}
\sum_{\ell\neq0}
\|\partial_y\psi_\ell(t)\|_{L_y^\infty}
\lesssim
e^{-\delta_0\nu^{1/2}t}
\mathcal E_{\neq}(t)^{1/2}.
\end{equation}

On the other hand, the estimate \eqref{w1} gives
\begin{equation}\label{N1_T2_HL_interpolation}
\sum_{q\neq0}
|q|^{2m+1}\|\omega_q(t)\|_{L_y^2}^2
\lesssim
\nu^{-2/3}
e^{-2\delta_0\nu^{1/2}t}
\mathcal D_{\neq,1}(t)^{1/3}
\mathcal D_{\neq,4}(t)^{2/3}.
\end{equation}
Combining
\eqref{nonzero_mode_sobolev_estimate} and
\eqref{N1_T2_HL_interpolation}, we obtain
\begin{align}\label{N1_T2_HL_final}
\left|
\left.\mathcal{T}_2\right|_{\Omega_{\mathrm{HL}}}
\right|
\lesssim
\nu^{-2/3}
\mathcal E_{\neq}^{1/2}
\mathcal D_{\neq,1}^{1/3}
\mathcal D_{\neq,4}^{2/3}.
\end{align}
Here and below, the exponential factors are absorbed into the
definitions of $\mathcal{E}_{\neq}$ and
$\mathcal{D}_{\neq,j}$.

On $\Omega_{\mathrm{LH}}$, we have
$
|k|\simeq|\ell|
$
by \eqref{frequency_comparability_LH}. Consequently,
\begin{align*}
|k|^{2m+1}
\lesssim
|k|^{m+\frac14}|\ell|^{m+\frac34}.
\end{align*}
Using the one-dimensional Gagliardo--Nirenberg inequality and the
elliptic relation $\partial_y^2\psi_\ell=\omega_\ell+k^2\psi_\ell$, we
have
\begin{align*}
\|\partial_y\psi_\ell\|_{L_y^\infty}
&\lesssim
\|\partial_y\psi_\ell\|_{L_y^2}^{1/2}
\|\partial_y^2\psi_\ell\|_{L_y^2}^{1/2}
\lesssim
\|\partial_y\psi_\ell\|_{L_y^2}^{1/2}
\|\omega_\ell\|_{L_y^2}^{1/2}.
\end{align*}
It follows that
\begin{align*}
\left|
\left.\mathcal{T}_2\right|_{\Omega_{\mathrm{LH}}}
\right|
\lesssim&
e^{2\delta_0\nu^{1/2}t}
\sum_{(k,\ell)\in\Omega_{\mathrm{LH}}}
\Bigl(
|k|^{m+\frac14}\|\omega_k\|_{L_y^2}
\Bigr)
\Bigl(
|\ell|^{m+\frac34}
\|\omega_\ell\|_{L_y^2}^{1/2}
\|\partial_y\psi_\ell\|_{L_y^2}^{1/2}
\Bigr)
\|\omega_{k-\ell}\|_{L_y^2}.
\end{align*}
Applying the weighted discrete Cauchy--Schwarz and H\"older inequalities,
we obtain
\begin{align*}
\left|
\left.\mathcal{T}_2\right|_{\Omega_{\mathrm{LH}}}
\right|
\lesssim&
e^{2\delta_0\nu^{1/2}t}
\Big(
\sum_{k\neq0}
|k|^{2m+\frac12}
\|\omega_k\|_{L_y^2}^2
\Big)^{1/2}
\Big(
\sum_{q\neq0}
|q|^{2m}
\|\omega_q\|_{L_y^2}^2
\Big)^{1/2}
\nonumber\\
&\qquad\qquad\times
\Big(
\sum_{\ell\neq0}
|\ell|^{2m+2}
\|\omega_\ell\|_{L_y^2}^2
\Big)^{1/4}
\Big(
\sum_{\ell\neq0}
|\ell|^{2m+1}
\|\partial_y\psi_\ell\|_{L_y^2}^2
\Big)^{1/4}.
\end{align*}
The four factors on the right-hand side are controlled, respectively,
by $\mathcal{D}_{\neq,4}$, $\mathcal{E}_{\neq}$,
$\mathcal{D}_{\neq,1}$, and $\mathcal{D}_{\neq,3}$. More precisely,
using $|k|\ge1$ for $k\neq0$, we have
\begin{align*}
\Big(
\sum_{k\neq0}
|k|^{2m+\frac12}
\|\omega_k\|_{L_y^2}^2
\Big)^{1/2}
&\lesssim
\nu^{-1/4}
e^{-\delta_0\nu^{1/2}t}
\mathcal{D}_{\neq,4}^{1/2},&
\Big(
\sum_{q\neq0}
|q|^{2m}
\|\omega_q\|_{L_y^2}^2
\Big)^{1/2}
&\lesssim
e^{-\delta_0\nu^{1/2}t}
\mathcal{E}_{\neq}^{1/2},
\\
\Big(
\sum_{\ell\neq0}
|\ell|^{2m+2}
\|\omega_\ell\|_{L_y^2}^2
\Big)^{1/4}
&\lesssim
\nu^{-1/4}
e^{-\frac12\delta_0\nu^{1/2}t}
\mathcal{D}_{\neq,1}^{1/4},&
\Big(
\sum_{\ell\neq0}
|\ell|^{2m+1}
\|\partial_y\psi_\ell\|_{L_y^2}^2
\Big)^{1/4}
&\lesssim
e^{-\frac12\delta_0\nu^{1/2}t}
\mathcal{D}_{\neq,3}^{1/4},
\end{align*}
from which, we conclude that
\begin{align}
\left|
\left.\mathcal{T}_2\right|_{\Omega_{\mathrm{LH}}}
\right|
\lesssim
\nu^{-1/2}
\mathcal{E}_{\neq}^{1/2}
\mathcal{D}_{\neq,4}^{1/2}
\mathcal{D}_{\neq,1}^{1/4}
\mathcal{D}_{\neq,3}^{1/4}.
\label{N1_T2_LH_final}
\end{align}

Finally, combining \eqref{N1_T2_HL_final} and
\eqref{N1_T2_LH_final},
we obtain
\begin{align*}
\left|\mathcal{T}_2\right|
\lesssim
\nu^{-2/3}\mathcal{E}_{\neq}^{1/2}
\Bigl(
\mathcal{D}_{\neq,1}^{1/3}
\mathcal{D}_{\neq,4}^{2/3}
+
\mathcal{D}_{\neq,4}^{1/2}
\mathcal{D}_{\neq,1}^{1/4}
\mathcal{D}_{\neq,3}^{1/4}
\Bigr).
\end{align*}

	Summing the  previous estimates, we recover the composite bound \eqref{N1}. This completes the proof of Lemma \ref{le_N1}.
\end{proof}

\subsection{Proof of Lemma \ref{le_N2}}

\begin{proof}
 By the definition of $\mathcal{N}_{\neq}^{(2)}$, we have
\begin{align}\label{N2_expansion}
\mathcal{N}_{\neq}^{(2)}
=&
-2\alpha\nu^{1/2}e^{2\delta_0\nu^{1/2}t}
\sum_{k\neq0}|k|^{2m-\frac12}
\operatorname{Re}
\Big\langle
\nabla_k\omega_k,
\nabla_k
\sum_{\ell\in\mathbb Z}
\partial_y\psi_\ell\,i(k-\ell)\omega_{k-\ell}
\Big\rangle
\nonumber\\
&+
2\alpha\nu^{1/2}e^{2\delta_0\nu^{1/2}t}
\sum_{k\neq0}|k|^{2m-\frac12}
\operatorname{Re}
\Big\langle
\nabla_k\omega_k,
\nabla_k
\sum_{\ell\in\mathbb Z}
i\ell\psi_\ell\,\partial_y\omega_{k-\ell}
\Big\rangle .
\end{align}
We therefore decompose
\begin{align*}
\mathcal{N}_{\neq}^{(2)}
=
\mathcal{N}_{0,\neq}^{(2)}
+
\mathcal{N}_{\neq,0}^{(2)}
+
\mathcal{N}_{\neq,\neq}^{(2)}.
\end{align*}

\smallskip
\noindent\textbf{Mean--fluctuation interaction ($\ell=0$).}

When $\ell=0$, the velocity factor is the zero mode and the vorticity
factor is the nonzero mode. Since $\partial_y\psi_0$ is real-valued, the
terms in which $\nabla_k$ does not fall on $\partial_y\psi_0$ have zero
real part. Thus,
\begin{align*}
\mathcal{N}_{0,\neq}^{(2)}
=
-2\alpha\nu^{1/2}e^{2\delta_0\nu^{1/2}t}
\sum_{k\neq0}|k|^{2m-\frac12}
\operatorname{Re}
\Big\langle
\partial_y\omega_k,
\partial_y^2\psi_0\,ik\omega_k
\Big\rangle .
\end{align*}
Indeed, the contribution involving
$\partial_y\psi_0\,ik\partial_y\omega_k$ is purely imaginary after
taking the real part.

By Lemma \ref{psi_0}, we have
\begin{align*}
\|\partial_y^2\psi_0\|_{L_y^\infty}
=
\|\omega_0\|_{L_y^\infty}
\lesssim
\nu^{-1/12}\mathcal{E}_0^{1/2}.
\end{align*}
Therefore,
\begin{align*}
\left|
\mathcal{N}_{0,\neq}^{(2)}
\right|
&\lesssim
\nu^{5/12}\mathcal{E}_0^{1/2}
e^{2\delta_0\nu^{1/2}t}
\sum_{k\neq0}
|k|^{2m+\frac12}
\|\partial_y\omega_k\|_{L_y^2}
\|\omega_k\|_{L_y^2}.
\end{align*}
We distribute the Fourier weights according to
\begin{align*}
|k|^{2m+\frac12}
=
|k|^m
\bigl(|k|^{m+1}\bigr)^{1/3}
\bigl(|k|^{m+\frac14}\bigr)^{2/3}.
\end{align*}
Consequently, the discrete H\"older inequality gives
\begin{align}
\left|
\mathcal{N}_{0,\neq}^{(2)}
\right|
\lesssim&
\nu^{5/12}\mathcal{E}_0^{1/2}
e^{2\delta_0\nu^{1/2}t}
\sum_{k\neq0}
\bigl(|k|^m\|\partial_y\omega_k\|_{L_y^2}\bigr)
\bigl(|k|^{m+1}\|\omega_k\|_{L_y^2}\bigr)^{1/3}
\bigl(|k|^{m+\frac14}\|\omega_k\|_{L_y^2}\bigr)^{2/3}
\nonumber\\
\lesssim&
\nu^{-1/2}\mathcal{E}_0^{1/2}
\mathcal{D}_{\neq,1}^{2/3}
\mathcal{D}_{\neq,4}^{1/3}.
\label{N2_mean_fluctuation_final}
\end{align}
In the last inequality, we used the definitions of
$\mathcal{D}_{\neq,1}$ and $\mathcal{D}_{\neq,4}$, the fact that
$|k|\ge1$ for $k\neq0$, and the exponential weights built into these
functionals.

\smallskip
\noindent\textbf{Fluctuation--mean interaction $(\ell=k)$.}

When $\ell=k$, the factor $k-\ell$ vanishes in the first term of
\eqref{N2_expansion}. Hence, integrating by parts in $y$ and using the definition of $\Delta_k$, we
obtain
\begin{align*}
\mathcal{N}_{\neq,0}^{(2)}
=
-2\alpha\nu^{1/2}e^{2\delta_0\nu^{1/2}t}
\sum_{k\neq0}|k|^{2m-\frac12}
\operatorname{Re}
\Big\langle
\Delta_k\omega_k,
ik\psi_k\,\partial_y\omega_0
\Big\rangle .
\end{align*}
Using
$
\|\partial_y\omega_0\|_{L_y^2}
\lesssim
\nu^{-1/6}\mathcal{E}_0^{1/2}
$
and the one-dimensional Gagliardo--Nirenberg inequality
\begin{align*}
\|\psi_k\|_{L_y^\infty}
\lesssim
\|\psi_k\|_{L_y^2}^{1/2}
\|\partial_y\psi_k\|_{L_y^2}^{1/2},
\end{align*}
we find
\begin{align*}
\left|
\mathcal{N}_{\neq,0}^{(2)}
\right|
\lesssim&
\nu^{1/3}\mathcal{E}_0^{1/2}
e^{2\delta_0\nu^{1/2}t}
\sum_{k\neq0}
\bigl(
|k|^{m-\frac14}
\|\Delta_k\omega_k\|_{L_y^2}
\bigr)
\bigl(
|k|^{m+\frac32}
\|\psi_k\|_{L_y^2}
\bigr)^{1/2}
\bigl(
|k|^{m+\frac12}
\|\partial_y\psi_k\|_{L_y^2}
\bigr)^{1/2}.
\end{align*}
Applying the discrete H\"older inequality and using the definitions of
$\mathcal{D}_{\neq,2}$ and $\mathcal{D}_{\neq,3}$, we obtain
\begin{align}
\left|
\mathcal{N}_{\neq,0}^{(2)}
\right|
&\lesssim
\nu^{-1/2}\mathcal{E}_0^{1/2}
\mathcal{D}_{\neq,2}^{1/2}
\mathcal{D}_{\neq,3}^{1/2}.
\label{N2_fluctuation_mean_final}
\end{align}

\smallskip
\noindent\textbf{Fluctuation--fluctuation interaction
$(\ell\neq0,\ k-\ell\neq0)$.}

For the remaining contribution, we have
\begin{align*}
\mathcal{N}_{\neq,\neq}^{(2)}
=
\mathcal{T}_3+\mathcal{T}_4,
\end{align*}
where
\begin{align*}
\mathcal{T}_3
=&
-2\alpha\nu^{1/2}e^{2\delta_0\nu^{1/2}t}
\sum_{\substack{k\neq0\\ \ell\neq0,\ k-\ell\neq0}}
|k|^{2m-\frac12}
\operatorname{Re}
\Big\langle
\Delta_k\omega_k,
\partial_y\psi_\ell\,i(k-\ell)\omega_{k-\ell}
\Big\rangle,\\
\mathcal{T}_4
=&
2\alpha\nu^{1/2}e^{2\delta_0\nu^{1/2}t}
\sum_{\substack{k\neq0\\ \ell\neq0,\ k-\ell\neq0}}
|k|^{2m-\frac12}
\operatorname{Re}
\Big\langle
\Delta_k\omega_k,
i\ell\psi_\ell\,\partial_y\omega_{k-\ell}
\Big\rangle .
\end{align*}
We first estimate $\mathcal{T}_3$. According to the frequency
decomposition \eqref{frequency_partition}, we write
\begin{align*}
\mathcal{T}_3
&=
\left.\mathcal{T}_3\right|_{\Omega_{\mathrm{LH}}}
+
\left.\mathcal{T}_3\right|_{\Omega_{\mathrm{HL}}}.
\end{align*}

On $\Omega_{\mathrm{LH}}$, the frequency comparison
\eqref{frequency_comparability_LH} gives
$
|k|\simeq|\ell|.
$
Using the one-dimensional Sobolev inequality
\begin{align*}
\|\partial_y\psi_\ell\|_{L_y^\infty}
\lesssim
\|\partial_y\psi_\ell\|_{L_y^2}^{1/2}
\|\partial_y^2\psi_\ell\|_{L_y^2}^{1/2},
\end{align*}
together with the weighted discrete H\"older inequality, we obtain
\begin{align}
\label{N2_T3_LH}
\left|
\left.\mathcal{T}_3\right|_{\Omega_{\mathrm{LH}}}
\right|
\lesssim&
\nu^{1/2}e^{2\delta_0\nu^{1/2}t}
\sum_{(k,\ell)\in\Omega_{\mathrm{LH}}}
\bigl(
|k|^{m-\frac14}
\|\Delta_k\omega_k\|_{L_y^2}
\bigr)
\nonumber\\
&\quad\times
\bigl(
|\ell|^{\frac m2+\frac18}
\|\partial_y^2\psi_\ell\|_{L_y^2}^{1/2}
\bigr)
\bigl(
|\ell|^{\frac m2+\frac14}
\|\partial_y\psi_\ell\|_{L_y^2}^{1/2}
\bigr)
|k-\ell|^{1/2}
\|\omega_{k-\ell}\|_{L_y^2}
\nonumber\\
\lesssim&
\nu^{-1/2}
\mathcal{E}_{\neq}^{1/2}
\mathcal{D}_{\neq,2}^{1/2}
\mathcal{D}_{\neq,3}^{1/4}
\mathcal{D}_{\neq,4}^{1/4}.
\end{align}
In the last inequality, we used the elliptic relation
$
\omega_\ell=\Delta_\ell\psi_\ell,
$
the definitions of $\mathcal{D}_{\neq,2}$,
$\mathcal{D}_{\neq,3}$, and $\mathcal{D}_{\neq,4}$, as well as the
weighted discrete convolution estimate. The exponential factors are
absorbed by the exponential weights in the definitions of
$\mathcal{E}_{\neq}$ and $\mathcal{D}_{\neq,j}$.

On $\Omega_{\mathrm{HL}}$, the frequency comparison
\eqref{frequency_comparability_HL} yields
$|k|\lesssim|k-\ell|.$ Hence, the principal output-frequency weight can be transferred to the
$k-\ell$-frequency. Using the $L_y^\infty$ estimate for
$\partial_y\psi_\ell$, the weighted discrete convolution inequality,
and the interpolation estimate \eqref{w2}, we obtain
\begin{align}
\label{N2_T3_HL}
\left|
\left.\mathcal{T}_3\right|_{\Omega_{\mathrm{HL}}}
\right|
\lesssim&
\nu^{1/2}e^{2\delta_0\nu^{1/2}t}
\sum_{(k,\ell)\in\Omega_{\mathrm{HL}}}
\bigl(
|k|^{m-\frac14}
\|\Delta_k\omega_k\|_{L_y^2}
\bigr)
\|\partial_y\psi_\ell\|_{L_y^\infty}
\bigl(
|k-\ell|^{m+\frac34}
\|\omega_{k-\ell}\|_{L_y^2}
\bigr)
\nonumber\\
\lesssim&
\nu^{-2/3}
\mathcal{E}_{\neq}^{1/2}
\mathcal{D}_{\neq,2}^{1/2}
\mathcal{D}_{\neq,1}^{1/3}
\mathcal{D}_{\neq,4}^{1/6}.
\end{align}
Combining \eqref{N2_T3_LH} and \eqref{N2_T3_HL}, we conclude that
\begin{align}
\label{N2_T3_final}
\left|\mathcal{T}_3\right|
\lesssim&
\nu^{-1/2}
\mathcal{E}_{\neq}^{1/2}
\mathcal{D}_{\neq,2}^{1/2}
\mathcal{D}_{\neq,3}^{1/4}
\mathcal{D}_{\neq,4}^{1/4}
+
\nu^{-2/3}
\mathcal{E}_{\neq}^{1/2}
\mathcal{D}_{\neq,2}^{1/2}
\mathcal{D}_{\neq,1}^{1/3}
\mathcal{D}_{\neq,4}^{1/6}.
\end{align}

We next estimate $\mathcal{T}_4$. We again use the partition
\eqref{frequency_partition} and write
\begin{align*}
\mathcal{T}_4
&=
\left.\mathcal{T}_4\right|_{\Omega_{\mathrm{LH}}}
+
\left.\mathcal{T}_4\right|_{\Omega_{\mathrm{HL}}}.
\end{align*}

On $\Omega_{\mathrm{LH}}$, we have
$
|k|\simeq|\ell|.
$
Applying the one-dimensional Gagliardo--Nirenberg inequality
\begin{align*}
\|\psi_\ell\|_{L_y^\infty}
\lesssim
\|\psi_\ell\|_{L_y^2}^{1/2}
\|\partial_y\psi_\ell\|_{L_y^2}^{1/2},
\end{align*}
and then using the weighted discrete H\"older inequality, we obtain
\begin{align}
\label{N2_T4_LH}
\left|
\left.\mathcal{T}_4\right|_{\Omega_{\mathrm{LH}}}
\right|
\lesssim&
\nu^{1/2}e^{2\delta_0\nu^{1/2}t}
\sum_{(k,\ell)\in\Omega_{\mathrm{LH}}}
\bigl(
|k|^{m-\frac14}
\|\Delta_k\omega_k\|_{L_y^2}
\bigr)
\nonumber\\
&\quad\times
\bigl(
|\ell|^{m+\frac32}
\|\psi_\ell\|_{L_y^2}
\bigr)^{1/2}
\bigl(
|\ell|^{m+\frac12}
\|\partial_y\psi_\ell\|_{L_y^2}
\bigr)^{1/2}
\|\partial_y\omega_{k-\ell}\|_{L_y^2}
\nonumber\\
\lesssim&
\nu^{-1/2}
\mathcal{E}_{\neq}^{1/2}
\mathcal{D}_{\neq,2}^{1/2}
\mathcal{D}_{\neq,3}^{1/2}.
\end{align}

On $\Omega_{\mathrm{HL}}$, we have
\begin{align*}
|\ell|\lesssim|k-\ell|,
\qquad
|k|\lesssim|k-\ell|.
\end{align*}
We therefore transfer the relevant frequency weights to the
$k-\ell$-frequency. Using the Gagliardo--Nirenberg inequality for
$\psi_\ell$, we obtain
\begin{align}
\label{N2_T4_HL}
\left|
\left.\mathcal{T}_4\right|_{\Omega_{\mathrm{HL}}}
\right|
\lesssim&
\nu^{1/2}e^{2\delta_0\nu^{1/2}t}
\sum_{(k,\ell)\in\Omega_{\mathrm{HL}}}
\bigl(
|k|^{m-\frac14}
\|\Delta_k\omega_k\|_{L_y^2}
\bigr)
\nonumber\\
&\quad\times
\bigl(
|k-\ell|^{m-\frac14}
\|\partial_y\omega_{k-\ell}\|_{L_y^2}
\bigr)
|\ell|
\bigl(
\|\psi_\ell\|_{L_y^2}
\|\partial_y\psi_\ell\|_{L_y^2}
\bigr)^{1/2}
\nonumber\\
\lesssim&
\nu^{-1/2}
\mathcal{E}_{\neq}^{1/2}
\mathcal{D}_{\neq,2}^{1/2}
\mathcal{D}_{\neq,3}^{1/2}.
\end{align}
Combining \eqref{N2_T4_LH} and \eqref{N2_T4_HL}, we arrive at
\begin{align}
\label{N2_T4_final}
\left|\mathcal{T}_4\right|
\lesssim
\nu^{-1/2}
\mathcal{E}_{\neq}^{1/2}
\mathcal{D}_{\neq,2}^{1/2}
\mathcal{D}_{\neq,3}^{1/2}.
\end{align}

We now combine the estimates for the three interaction types. From
\eqref{N2_mean_fluctuation_final} and
\eqref{N2_fluctuation_mean_final}, we have
\begin{align*}
\left|
\mathcal{N}_{0,\neq}^{(2)}
\right|
+
\left|
\mathcal{N}_{\neq,0}^{(2)}
\right|
\lesssim&
\nu^{-1/2}\mathcal{E}^{1/2}
\Bigl(
\mathcal{D}_{\neq,1}^{2/3}
\mathcal{D}_{\neq,4}^{1/3}
+
\mathcal{D}_{\neq,2}^{1/2}
\mathcal{D}_{\neq,3}^{1/2}
\Bigr),
\end{align*}
where we used $\mathcal{E}_0\le\mathcal{E}$.

Moreover, by \eqref{N2_T3_final} and \eqref{N2_T4_final},  we obtain
\begin{align*}
\left|
\mathcal{N}_{\neq,\neq}^{(2)}
\right|
\lesssim&
\nu^{-1/2}\mathcal{E}_{\neq}^{1/2}
\mathcal{D}_{\neq,2}^{1/2}
\left(
\mathcal{D}_{\neq,3}^{1/2}
+
\mathcal{D}_{\neq,4}^{1/2}
\right)
+
\nu^{-2/3}\mathcal{E}_{\neq}^{1/2}
\mathcal{D}_{\neq,2}^{1/2}
\mathcal{D}_{\neq,1}^{1/3}
\mathcal{D}_{\neq,4}^{1/6}.
\end{align*}

Finally, since $0<\nu<1$ implies
$
\nu^{-1/2}\le\nu^{-2/3},
$
and since $\mathcal{E}_0,\mathcal{E}_{\neq}\le\mathcal{E}$, the preceding
estimates give
\begin{align*}
\left|
\mathcal{N}_{\neq}^{(2)}
\right|
\lesssim
\nu^{-2/3}\mathcal{E}^{1/2}
\Bigl(
&\mathcal{D}_{\neq,1}^{2/3}
\mathcal{D}_{\neq,4}^{1/3}
+
\mathcal{D}_{\neq,2}^{1/2}
\mathcal{D}_{\neq,3}^{1/2}
+
\mathcal{D}_{\neq,2}^{1/2}
\mathcal{D}_{\neq,4}^{1/2}
+
\mathcal{D}_{\neq,2}^{1/2}
\mathcal{D}_{\neq,1}^{1/3}
\mathcal{D}_{\neq,4}^{1/6}
\Bigr).
\end{align*}
This is precisely the estimate asserted in \eqref{N2}, and the proof is
complete.
\end{proof}

\subsection{Proof of Lemma \ref{le_N3}}

\begin{proof}
We first decompose the mixed contribution into
\begin{align*}
\mathcal{N}_{\neq}^{(3)}
&=
\mathcal{N}_{\neq}^{(3,1)}
+
\mathcal{N}_{\neq}^{(3,2)},
\end{align*}
where
\begin{align}
\label{N3_first_component}
\mathcal{N}_{\neq}^{(3,1)}
=&
\beta e^{2\delta_0\nu^{1/2}t}
\sum_{k\neq0}|k|^{2m-1}
\operatorname{Re}
\Big\langle
\partial_y\mathcal N_k,
iky\omega_k
\Big\rangle ,
\\
\label{N3_second_component}
\mathcal{N}_{\neq}^{(3,2)}
=&
\beta e^{2\delta_0\nu^{1/2}t}
\sum_{k\neq0}|k|^{2m-1}
\operatorname{Re}
\Big\langle
iky\mathcal N_k,
\partial_y\omega_k
\Big\rangle .
\end{align}

After integrating by parts in $y$, the estimate of
$\mathcal N_{\neq}^{(3)}$ reduces to the estimate of
$\mathcal N_{\neq}^{(3,1)}$ and the auxiliary unweighted term
\begin{align*}
\mathcal R_{\neq}
:=
\beta e^{2\delta_0\nu^{1/2}t}
\sum_{k\neq0}|k|^{2m-1}
\operatorname{Re}
\Big\langle
\mathcal N_k,
ik\omega_k
\Big\rangle .
\end{align*}
More precisely,
\begin{align*}
\left|\mathcal N_{\neq}^{(3)}\right|
\lesssim
\left|\mathcal N_{\neq}^{(3,1)}\right|
+
\left|\mathcal R_{\neq}\right|.
\end{align*}
The term $\mathcal N_{\neq}^{(3,2)}$ is treated in the same manner after
the corresponding integration by parts.

Expanding the nonlinear term in \eqref{N3_first_component}, we obtain
\begin{align}
\label{N3_first_expansion}
\mathcal N_{\neq}^{(3,1)}
=&
\beta e^{2\delta_0\nu^{1/2}t}
\sum_{k\neq0}|k|^{2m-1}
\operatorname{Re}
\Bigg\langle
iky\omega_k,
\partial_y
\sum_{\ell\in\mathbb Z}
\partial_y\psi_\ell\,i(k-\ell)\omega_{k-\ell}
\Bigg\rangle
\nonumber\\
&-
\beta e^{2\delta_0\nu^{1/2}t}
\sum_{k\neq0}|k|^{2m-1}
\operatorname{Re}
\Bigg\langle
iky\omega_k,
\partial_y
\sum_{\ell\in\mathbb Z}
i\ell\psi_\ell\,\partial_y\omega_{k-\ell}
\Bigg\rangle .
\end{align}

According to the interaction convention, we write
\begin{align*}
\mathcal N_{\neq}^{(3,1)}
&=
\mathcal N_{0,\neq}^{(3,1)}
+
\mathcal N_{\neq,0}^{(3,1)}
+
\mathcal N_{\neq,\neq}^{(3,1)}.
\end{align*}

\smallskip
\noindent\textbf{ Mean--fluctuation interaction $(\ell=0)$.}

For $\ell=0$, only the first term in
\eqref{N3_first_expansion} contributes. Hence,
\begin{align*}
\mathcal N_{0,\neq}^{(3,1)}
&=
\beta e^{2\delta_0\nu^{1/2}t}
\sum_{k\neq0}|k|^{2m-1}
\operatorname{Re}
\Bigg\langle
iky\omega_k,
\partial_y\bigl(
\partial_y\psi_0\,ik\omega_k
\bigr)
\Bigg\rangle .
\end{align*}
Integrating by parts and symmetrizing with respect to the $y$-variable
gives
\begin{align*}
\mathcal N_{0,\neq}^{(3,1)}
=&
-\frac{\beta}{2}
e^{2\delta_0\nu^{1/2}t}
\sum_{k\neq0}|k|^{2m-1}
\operatorname{Re}
\Big\langle
ik\omega_k,
\partial_y\psi_0\,ik\omega_k
\Big\rangle
\nonumber\\
&+
\frac{\beta}{2}
e^{2\delta_0\nu^{1/2}t}
\sum_{k\neq0}|k|^{2m-1}
\operatorname{Re}
\Big\langle
iky\omega_k,
\partial_y^2\psi_0\,ik\omega_k
\Big\rangle
\nonumber\\
=:&
\mathcal T_5+\mathcal T_6.
\end{align*}
		The second component $\mathcal{T}_{6}$ can be estimated directly. Indeed,  by Lemma \ref{psi_0}, we have
		\begin{align}\label{t6guji}
		|\mathcal{T}_{6}|
		&\lesssim
		e^{2\delta_0\nu^{1/2}t}
		\sum_{k\neq 0}
		|k|^{2m+1}
		\|y\omega_k\|_{L^2}
		\|\omega_k\|_{L^2}
		\|\partial_y^2\psi_0\|_{L^\infty}\nn\\
		&\lesssim
		\nu^{-1/12}\mathcal{E}_0^{1/2}
		e^{2\delta_0\nu^{1/2}t}
		\sum_{k\neq 0}
		\left(
		|k|^m|k|^{1/2}\|y\omega_k\|_{L^2}
		\right)
		\left(
		|k|^m|k|^{1/4}\|\omega_k\|_{L^2}
		\right)^{2/3}
		\left(
		|k|^m|k|\|\omega_k\|_{L^2}
		\right)^{1/3}
		\nn\\
		&\lesssim
		\nu^{-2/3}
		\mathcal{E}_0^{1/2}
		\mathcal{D}_{\neq,3}^{1/2}
		\mathcal{D}_{\neq,4}^{1/3}
		\mathcal{D}_{\neq,1}^{1/6}.
		\end{align}

		To bound the first component $\mathcal{T}_{5}$, we must control the evolution of the mean streamwise velocity $u_0^1$. Recall the momentum equation for $u^1=u_0^1+u_{\neq}^1$:
		\begin{align*}
		\partial_t u^1
		+
		y^2\partial_x u^1
		+
		2yu^2
		+
		u\cdot\nabla u^1
		+
		\partial_x p
		-
		\nu\Delta u^1
		=
		0.
		\end{align*}
		Projecting onto the zero mode yields
		\begin{align}\label{eq_u0}
		\partial_t u_0^1
		-
		\nu\partial_y^2 u_0^1
		=
		-\partial_y P_0(u_{\neq}^1u_{\neq}^2).
		\end{align}
		Taking an $L^2$-inner product of \eqref{eq_u0} with $u_0^1$ and applying Young's inequality, we obtain
		\begin{align*}
		\frac12\frac{\mathrm{d}}{\mathrm{d}t}\|u_0^1\|_{L^2_y}^2
		+
		\nu\|\partial_y u_0^1\|_{L^2_y}^2
		&=
		\int_{\mathbb{R}}
		P_0(u_{\neq}^1u_{\neq}^2)\,
		\partial_y u_0^1\,dy\\
		&\le \frac{\nu}{2}\|\partial_y u_0^1\|_{L^2_y}^2
		+
		\frac{1}{2\nu}
		\|P_0(u_{\neq}^1u_{\neq}^2)\|_{L^2_y}^2 .
		\end{align*}
		
		Using the bound $\|P_0(u_{\neq}^1u_{\neq}^2)\|_{L^2_y}^2 \lesssim \|u_{\neq}^1\|_{L^\infty_{x,y}}^2 \|u_{\neq}^2\|_{L^2_{x,y}}^2$ and integrating in time($ 0 \le s \le t$) gives
		\begin{align}
		\|u_0^1(s)\|_{L^2_y}^2
		\lesssim
		\|u_{\mathrm{in,0}}\|_{L^2}^2
		+
		\nu^{-1}
		\sup_{0\le \tau\le s}
		\|u_{\neq}^2(\tau)\|_{L^2_{x,y}}^2
		\int_0^s
		\|u_{\neq}^1(\tau)\|_{L^\infty_{x,y}}^2\,d\tau .
		\end{align}
		
		For the fluctuation $u_{\neq}^2$, we have
		\begin{align*}
		\|u_{\neq}^2\|_{L^2_{x,y}}^2
		\lesssim
		\sum_{k\neq 0}\|\omega_k\|_{L^2_y}^2
		\lesssim
		\mathcal{E}_{\neq}.
		\end{align*}
		Moreover, the one-dimensional Sobolev inequality implies
		\begin{align}\label{u1_neq}
		\|u_{\neq}^1\|_{L^\infty_{x,y}}^2
		&\lesssim
		\sup_y
		\sum_{k\neq 0}
		|k|^{2m}|\partial_y\psi_k(y)|^2 \nonumber\\
		&\lesssim
		\sum_{k\neq 0}
		|k|^{2m}
		\|\partial_y\psi_k\|_{L^2_y}
		\|\partial_y^2\psi_k\|_{L^2_y}
		\nonumber\\
		&\lesssim
		\Big(
		\sum_{k\neq 0}
		|k|^{2m}\|\nabla_k\psi_k\|_{L^2_y}^2
		\Big)^{1/2}
		\Big(
		\sum_{k\neq 0}
		|k|^{2m}\|\omega_k\|_{L^2_y}^2
		\Big)^{1/2}.
		\end{align}
		
		Returning to $\mathcal{T}_{5}$ and applying \eqref{w1} from Lemma~\ref{w_neq}, we deduce
		\begin{align}\label{nkjg}
		|\mathcal{T}_{5}|
		&\lesssim
		e^{2\delta_0\nu^{1/2}t}
		\sum_{k\neq 0}
		|k|^{2m+1}
		\|\omega_k\|_{L^2}^2
		\|u^1_0\|_{L^\infty}\nn\\
		&\lesssim
		\nu^{-2/3}
		\mathcal{D}_{\neq,1}^{1/3}
		\mathcal{D}_{\neq,4}^{2/3}
		\|u_0^1\|_{L^2}^{1/2}
		\|\omega_0\|_{L^2}^{1/2}\nn\\
		&\lesssim \nu^{-2/3} \mathcal{E}_0^{1/4}
		\mathcal{D}_{\neq,1}^{1/3}
		\mathcal{D}_{\neq,4}^{2/3}
		\|u_0^1\|_{L^2}^{1/2}\nn\\
		& \lesssim \nu^{-2/3} \mathcal{E}_0^{1/4}
		\mathcal{D}_{\neq,1}^{1/3}
		\mathcal{D}_{\neq,4}^{2/3} \left( \|u_{\mathrm{in,0}}\|_{L^2}^2
		+
		\nu^{-1}
		\sup_{0\le \tau\le s}
		\|u_{\neq}^2(\tau)\|_{L^2_{x,y}}^2
		\int_0^s
		\|u_{\neq}^1(\tau)\|_{L^\infty_{x,y}}^2\,d\tau \right)^{1/4}.
		\end{align}

		Under the smallness assumption $\|u_{\mathrm{in,0}}\|_{L^2_{y}} \le \varepsilon \nu^{2/3}$ of Theorem \ref{thm1}, the contribution involving the initial data is controlled via Young's inequality as
		\begin{align}\label{N3_T5_initial_data}
		\nu^{-2/3}
		\mathcal{E}_{0}^{1/4}
		\mathcal{D}_{\neq,1}^{1/3}
		\mathcal{D}_{\neq,4}^{2/3}
		\|u_{\mathrm{in},0}\|_{L^2}^{1/2}
		\le
		2\delta_0
		\left(
		\mathcal{D}_{\neq,1}+\mathcal{D}_{\neq,4}
		\right)
		+
		C\nu^{-2/3}
		\mathcal{E}_0^{1/2}
		\mathcal{D}_{\neq,1}^{1/3}
		\mathcal{D}_{\neq,4}^{2/3}.
		\end{align}
		For the nonlinear contribution in \eqref{nkjg}, the estimate \eqref{u1_neq} yields
		\begin{align}\label{N3_T5_nonlinear}
		&\nu^{-2/3-1/4}
		\mathcal{E}_{0}^{1/4}
		\mathcal{D}_{\neq,1}^{1/3}
		\mathcal{D}_{\neq,4}^{2/3}
		\sup_{0\le \tau\le s}
		\|u^2_{\neq}(\tau)\|_{L^2_{x,y}}^{1/2}
		\left(
		\int_0^s
		\|u_{\neq}^1(\tau)\|_{L^\infty_{x,y}}^2\,d\tau
		\right)^{1/4}
		\\
		&\quad\lesssim
		\nu^{-2/3-1/4}
		\mathcal{E}^{1/2}
		\mathcal{D}_{\neq,1}^{1/3}
		\mathcal{D}_{\neq,4}^{2/3}
		\left(
		\int_0^s
		\left(
		\sum_{k\neq 0}
		|k|^{2m}\|\nabla_k\psi_k\|_{L^2}^2
		\right)^{1/2}
		\left(
		\sum_{k\neq 0}
		|k|^{2m}\|\omega_k\|_{L^2}^2
		\right)^{1/2}
		d\tau
		\right)^{1/4}.\nn
		\end{align}

Combining \eqref{N3_T5_initial_data}, \eqref{N3_T5_nonlinear} and \eqref{t6guji},
 we establish the bound for the mean--fluctuation interaction:
		\begin{align*}
			\mathcal{N}_{0,\neq}^{(3,1)} &\lesssim \nu^{-2/3}
			\mathcal{E}_0^{1/2}
			\mathcal{D}_{\neq,3}^{1/2}
			\mathcal{D}_{\neq,4}^{1/3}
			\mathcal{D}_{\neq,1}^{1/6} + 2\delta_0
			\left(
			\mathcal{D}_{\neq,1}+\mathcal{D}_{\neq,4}
			\right)
			+
			C\nu^{-2/3}
			\mathcal{E}_0^{1/2}
			\mathcal{D}_{\neq,1}^{2/3}
			\mathcal{D}_{\neq,4}^{1/3}\\
			&\quad +\nu^{-2/3-1/4}
			\mathcal{E}^{1/2}
			\mathcal{D}_{\neq,1}^{1/3}
			\mathcal{D}_{\neq,4}^{2/3}
			\left(
			\int_0^s
			\left(
			\sum_{k\neq 0}
			|k|^{2m}\|\nabla_k\psi_k\|_{L^2}^2
			\right)^{1/2}
			\left(
			\sum_{k\neq 0}
			|k|^{2m}\|\omega_k\|_{L^2}^2
			\right)^{1/2}
			d\tau
			\right)^{1/4}.
		\end{align*}

\smallskip
\noindent\textbf{Fluctuation--mean interaction $(\ell=k)$.}

For $\ell=k$, the first term in
\eqref{N3_first_expansion} vanishes because $k-\ell=0$. Thus,
\begin{align*}
\mathcal N_{\neq,0}^{(3,1)}
&=
-\beta e^{2\delta_0\nu^{1/2}t}
\sum_{k\neq0}|k|^{2m-1}
\operatorname{Re}
\Bigg\langle
iky\omega_k,
\partial_y\bigl(
ik\psi_k\,\partial_y\omega_0
\bigr)
\Bigg\rangle .
\end{align*}
After integration by parts,
\begin{align*}
\mathcal N_{\neq,0}^{(3,1)}
=&
\beta e^{2\delta_0\nu^{1/2}t}
\sum_{k\neq0}|k|^{2m-1}
(\operatorname{Re}
\Big(\big\langle
ik\omega_k,
ik\psi_k\,\partial_y\omega_0
\Big\rangle+\operatorname{Re}
\Big\langle
iky\partial_y\omega_k,
ik\psi_k\,\partial_y\omega_0
\big\rangle \Big).
\end{align*}
Using
\begin{align*}
\|\partial_y\omega_0\|_{L_y^2}
\lesssim
\nu^{-1/6}\mathcal E_0^{1/2},
\qquad
\|\psi_k\|_{L_y^\infty}
\lesssim
|k|^{-1/2}\|\nabla_k\psi_k\|_{L_y^2},
\end{align*}
we obtain
\begin{align*}
\left|
\mathcal N_{\neq,0}^{(3,1)}
\right|
\lesssim&
\nu^{-1/2}
\mathcal E_0^{1/2}
\left(
\mathcal D_{\neq,4}^{1/2}
+
\mathcal D_{\neq,5}^{1/2}
\right)
\mathcal D_{\neq,3}^{1/2}
\lesssim
\nu^{-2/3}
\mathcal E^{1/2}\mathcal D_{\neq}.
\end{align*}

\smallskip
\noindent\textbf{Fluctuation--fluctuation interaction
$(\ell\neq0,\ k-\ell\neq0)$.}

For the nonzero--nonzero interaction, we write
\begin{align*}
\mathcal N_{\neq,\neq}^{(3,1)}
&=
\mathcal T_7+\mathcal T_8+\mathcal T_9+\mathcal T_{10},
\end{align*}
where
\begin{align*}
\mathcal T_7
&=
\beta e^{2\delta_0\nu^{1/2}t}
\sum_{\substack{k\neq0\\ \ell\neq0,\ k-\ell\neq0}}
|k|^{2m-1}
\operatorname{Re}
\Big\langle
iky\omega_k,
\partial_y^2\psi_\ell\,
i(k-\ell)\omega_{k-\ell}
\Big\rangle ,
\\
\mathcal T_8
&=
\beta e^{2\delta_0\nu^{1/2}t}
\sum_{\substack{k\neq0\\ \ell\neq0,\ k-\ell\neq0}}
|k|^{2m-1}
\operatorname{Re}
\Big\langle
iky\omega_k,
\partial_y\psi_\ell\,
i(k-\ell)\partial_y\omega_{k-\ell}
\Big\rangle ,
\\
\mathcal T_9
&=
-\beta e^{2\delta_0\nu^{1/2}t}
\sum_{\substack{k\neq0\\ \ell\neq0,\ k-\ell\neq0}}
|k|^{2m-1}
\operatorname{Re}
\Big\langle
ik\omega_k,
i\ell\psi_\ell\,
\partial_y\omega_{k-\ell}
\Big\rangle ,
\\
\mathcal T_{10}
&=
-\beta e^{2\delta_0\nu^{1/2}t}
\sum_{\substack{k\neq0\\ \ell\neq0,\ k-\ell\neq0}}
|k|^{2m-1}
\operatorname{Re}
\Big\langle
iky\partial_y\omega_k,
i\ell\psi_\ell\,
\partial_y\omega_{k-\ell}
\Big\rangle .
\end{align*}
We estimate these four terms separately, using throughout the frequency
partition \eqref{frequency_partition}.

\smallskip
\noindent\textbf{Estimate of $\mathcal T_9$ and $\mathcal T_{10}$.}
By the elementary weight inequality $|k|^m\lesssim|\ell|^m+|k-\ell|^m$ and the estimate
$
\|\ell\psi_\ell\|_{L_y^\infty}
\lesssim
|\ell|^{1/2}\|\nabla_\ell\psi_\ell\|_{L_y^2},
$
the weighted discrete convolution inequality gives
\begin{align*}
|\mathcal T_9|
&\lesssim
e^{2\delta_0\nu^{1/2}t}
\sum_{\substack{k\neq0\\ \ell\neq0,\ k-\ell\neq0}}
|k|^{2m}
\|\omega_k\|_{L_y^2}
\|\ell\psi_\ell\|_{L_y^\infty}
\|\partial_y\omega_{k-\ell}\|_{L_y^2}
\nonumber\\
&\lesssim
\nu^{-1/2}
\mathcal E_{\neq}^{1/2}
\mathcal D_{\neq,4}^{1/2}
\mathcal D_{\neq,3}^{1/2}.
\end{align*}
The same argument, with $y\partial_y\omega_k$ in place of
$\omega_k$, yields
\begin{align}
\label{N3_T10_estimate}
|\mathcal T_{10}|
&\lesssim
e^{2\delta_0\nu^{1/2}t}
\sum_{\substack{k\neq0\\ \ell\neq0,\ k-\ell\neq0}}
|k|^{2m}
\|y\partial_y\omega_k\|_{L_y^2}
\|\ell\psi_\ell\|_{L_y^\infty}
\|\partial_y\omega_{k-\ell}\|_{L_y^2}
\nonumber\\
&\lesssim
\nu^{-1/2}
\mathcal E_{\neq}^{1/2}
\mathcal D_{\neq,5}^{1/2}
\mathcal D_{\neq,3}^{1/2}.
\end{align}

\smallskip
\noindent\textbf{Estimate of $\mathcal T_7$.} Using the frequency partition \eqref{frequency_partition}, we write
\begin{align*}
\mathcal T_7
&=
\left.\mathcal T_7\right|_{\Omega_{\mathrm{LH}}}
+
\left.\mathcal T_7\right|_{\Omega_{\mathrm{HL}}}.
\end{align*}

On $\Omega_{\mathrm{HL}}$, the frequency relations
$
|k|\lesssim|k-\ell|$ and $
|\ell|\lesssim|k-\ell|
$
allow us to transfer the output-frequency weight to the
$k-\ell$-factor. Thus,
\begin{align}
\label{N3_T7_HL}
\left|
\left.\mathcal T_7\right|_{\Omega_{\mathrm{HL}}}
\right|
\lesssim&
e^{2\delta_0\nu^{1/2}t}
\sum_{\substack{(k,\ell)\in\Omega_{\mathrm{HL}}\\ k\neq0}}
\Bigl(
|k|^{m+\frac12}
\|y\omega_k\|_{L_y^2}
\Bigr)
\Bigl(
|k-\ell|^{m+\frac12}
\|\omega_{k-\ell}\|_{L_y^2}
\Bigr)
\|\partial_y^2\psi_\ell\|_{L_y^\infty}
\nonumber\\
\lesssim&
\nu^{-11/24}
\mathcal E_{\neq}^{1/2}
\mathcal D_{\neq,3}^{1/2}
\mathcal D_{\neq,1}^{1/6}
\mathcal D_{\neq,4}^{1/3}.
\end{align}
Here we used the discrete Young inequality and the following
estimate
\begin{align*}
\sum_{\ell\neq0}
\|\partial_y^2\psi_\ell\|_{L_y^\infty}
\lesssim
\nu^{-1/8}
e^{-\delta_0\nu^{1/2}t}
\mathcal E_{\neq}^{1/2}.
\end{align*}

On $\Omega_{\mathrm{LH}}$, we have $|k|\simeq|\ell|$. Hence the
principal weight can instead be assigned to the $\ell$-factor.
Using \eqref{w3}, the elliptic estimate
$
\|\partial_y^2\psi_\ell\|_{L_y^2}
\lesssim\|\omega_\ell\|_{L_y^2},
$
and the weighted convolution inequality, we find
\begin{align}
\label{N3_T7_LH}
\left|
\left.\mathcal T_7\right|_{\Omega_{\mathrm{LH}}}
\right|
\lesssim&
e^{2\delta_0\nu^{1/2}t}
\sum_{\substack{(k,\ell)\in\Omega_{\mathrm{LH}}\\ k\neq0}}
\Bigl(
|k|^{m+\frac12}
\|y\omega_k\|_{L_y^2}
\Bigr)
\Bigl(
|\ell|^{m+1}
\|\partial_y^2\psi_\ell\|_{L_y^2}
\Bigr)
\|\omega_{k-\ell}\|_{L_y^\infty}
\nonumber\\
\lesssim&
\nu^{-5/8}
\mathcal E_{\neq}^{1/2}
\mathcal D_{\neq,3}^{1/2}
\mathcal D_{\neq,1}^{1/2}.
\end{align}
Combining \eqref{N3_T7_HL} and \eqref{N3_T7_LH}, we obtain
\begin{align}\label{N3_T7_final}
|\mathcal T_7|
\lesssim&
\nu^{-11/24}
\mathcal E_{\neq}^{1/2}
\mathcal D_{\neq,3}^{1/2}
\mathcal D_{\neq,1}^{1/6}
\mathcal D_{\neq,4}^{1/3}
+
\nu^{-5/8}
\mathcal E_{\neq}^{1/2}
\mathcal D_{\neq,3}^{1/2}
\mathcal D_{\neq,1}^{1/2}.
\end{align}

\smallskip
\noindent\textbf{Estimate of $\mathcal T_8$.}
Using the decomposition $k=k-\ell+\ell$ and integrating by parts in  $y$ in the term involving the
the $k$-frequency,  we further expand $\mathcal{T}_8$ into four components:
	\begin{align*}
		\mathcal{T}_8
		& = -\beta e^{2\delta_0\nu^{1/2}t}
		\sum_{\substack{\ell,k\neq 0\\ k\neq \ell}}
		|k|^{2m-1}\Bigg\{
		\underbrace{\operatorname{Re}
		\Big\langle
		ik  \omega_k,
		\partial_y\psi_\ell\, ik   \omega_{k-\ell}
		\Big\rangle}_{\mathcal{T}_{8}^{(1)}}
+\underbrace{\operatorname{Re}
		\Big\langle
		ik y\partial_y \omega_k,
		\partial_y\psi_\ell\, ik   \omega_{k-\ell}
		\Big\rangle}_{\mathcal{T}_{8}^{(2)}}\Bigg\}\\
		&\quad -\beta e^{2\delta_0\nu^{1/2}t}
		\sum_{\substack{\ell,k\neq 0\\ k\neq \ell}}
		|k|^{2m-1}
		\Bigg\{\underbrace{\operatorname{Re}
		\Big\langle
		ik  y \omega_k,
		i\ell \partial_y\psi_\ell\,  \partial_{y} \omega_{k-\ell}
		\Big\rangle}_{\mathcal{T}_{8}^{(3)}}+\underbrace{\operatorname{Re}
		\Big\langle
		ik y \omega_k,
		\partial_y^2\psi_\ell\, ik   \omega_{k-\ell}
		\Big\rangle}_{\mathcal{T}_{8}^{(4)}}\Bigg\}.
		\end{align*}	
We next estimate the four components
$\mathcal{T}_{8}^{(1)},\ldots,\mathcal{T}_{8}^{(4)}$ using the common
frequency partition \eqref{frequency_partition}. For any
$j\in\{1,2,3,4\}$, we write
\begin{align}\label{N3_T8_frequency_decomposition}
\mathcal{T}_{8}^{(j)}
&=
\left.\mathcal{T}_{8}^{(j)}\right|_{\Omega_{\mathrm{LH}}}
+
\left.\mathcal{T}_{8}^{(j)}\right|_{\Omega_{\mathrm{HL}}}.
\end{align}
Recall that
\begin{align*}
|k|\simeq|\ell|
\quad\text{on }\Omega_{\mathrm{LH}},
\qquad
|k|\lesssim|k-\ell|,
\quad
|\ell|\lesssim|k-\ell|
\quad\text{on }\Omega_{\mathrm{HL}}.
\end{align*}

\smallskip
\noindent\textbf{Estimate of $\mathcal{T}_{8}^{(1)}$.}
By the definition of $\mathcal{T}_{8}^{(1)}$ and the Cauchy--Schwarz
inequality in $y$, we have
\begin{align*}
|\mathcal{T}_{8}^{(1)}|
\lesssim&
e^{2\delta_0\nu^{1/2}t}
\sum_{\substack{k\neq0\\ \ell\neq0,\ k-\ell\neq0}}
|k|^{2m+1}
\|\omega_k\|_{L_y^2}
\|\partial_y\psi_\ell\|_{L_y^\infty}
\|\omega_{k-\ell}\|_{L_y^2}.
\end{align*}

On $\Omega_{\mathrm{LH}}$, the relation
$|k|\simeq|\ell|$ allows us to distribute the Fourier weights between
the $k$- and $\ell$-factors. Using the one-dimensional
Gagliardo--Nirenberg inequality for $\partial_y\psi_\ell$, followed by
the weighted discrete H\"older and convolution inequalities, we obtain
\begin{align}
\label{N3_T81_LH}
\left|
\left.\mathcal{T}_{8}^{(1)}\right|_{\Omega_{\mathrm{LH}}}
\right|
\lesssim&
e^{2\delta_0\nu^{1/2}t}
\sum_{(k,\ell)\in\Omega_{\mathrm{LH}}}
\Bigl(
|k|^{m+\frac14}\|\omega_k\|_{L_y^2}
\Bigr)^{1/2}
\Bigl(
|k|^{m+1}\|\omega_k\|_{L_y^2}
\Bigr)^{1/2}
\nonumber\\
&\quad\times
\Bigl(
|\ell|^{m+\frac14}\|\omega_\ell\|_{L_y^2}
\Bigr)^{1/2}
\Bigl(
|\ell|^{m+\frac12}
\|\partial_y\psi_\ell\|_{L_y^2}
\Bigr)^{1/2}
\|\omega_{k-\ell}\|_{L_y^2}
\nonumber\\
\lesssim&
\nu^{-1/2}
\mathcal{E}_{\neq}^{1/2}
\mathcal{D}_{\neq,4}^{1/2}
\mathcal{D}_{\neq,1}^{1/4}
\mathcal{D}_{\neq,3}^{1/4}.
\end{align}

On $\Omega_{\mathrm{HL}}$, we use $|k|\lesssim|k-\ell|$ and transfer the output-frequency weight to the
$k-\ell$-factor. The weighted convolution inequality and the
interpolation estimate \eqref{w1} then give
\begin{align}
\label{N3_T81_HL}
\left|
\left.\mathcal{T}_{8}^{(1)}\right|_{\Omega_{\mathrm{HL}}}
\right|
\lesssim&
e^{2\delta_0\nu^{1/2}t}
\sum_{(k,\ell)\in\Omega_{\mathrm{HL}}}
\Bigl(
|k|^{m+\frac12}\|\omega_k\|_{L_y^2}
\Bigr)
\|\partial_y\psi_\ell\|_{L_y^\infty}
\Bigl(
|k-\ell|^{m+\frac12}
\|\omega_{k-\ell}\|_{L_y^2}
\Bigr)
\nonumber\\
\lesssim&
\nu^{-2/3}
\mathcal{E}_{\neq}^{1/2}
\mathcal{D}_{\neq,4}^{2/3}
\mathcal{D}_{\neq,1}^{1/3}.
\end{align}
Combining \eqref{N3_T81_LH} and \eqref{N3_T81_HL}, we obtain
\begin{align}
\label{N3_T81_final}
|\mathcal{T}_{8}^{(1)}|
\lesssim&
\nu^{-1/2}
\mathcal{E}_{\neq}^{1/2}
\mathcal{D}_{\neq,4}^{1/2}
\mathcal{D}_{\neq,1}^{1/4}
\mathcal{D}_{\neq,3}^{1/4}
+
\nu^{-2/3}
\mathcal{E}_{\neq}^{1/2}
\mathcal{D}_{\neq,4}^{2/3}
\mathcal{D}_{\neq,1}^{1/3}.
\end{align}

\smallskip
\noindent\textbf{Estimate of $\mathcal{T}_{8}^{(2)}$.}
By the definition of $\mathcal{T}_{8}^{(2)}$, we have
\begin{align*}
|\mathcal{T}_{8}^{(2)}|
\lesssim&
e^{2\delta_0\nu^{1/2}t}
\sum_{\substack{k\neq0\\ \ell\neq0,\ k-\ell\neq0}}
|k|^{2m+1}
\|y\partial_y\omega_k\|_{L_y^2}
\|\partial_y\psi_\ell\|_{L_y^\infty}
\|\omega_{k-\ell}\|_{L_y^2}.
\end{align*}

On $\Omega_{\mathrm{LH}}$, using
$|k|\simeq|\ell|$, the one-dimensional Sobolev inequality, and the
weighted convolution estimates, we obtain
\begin{align*}
\left|
\left.\mathcal{T}_{8}^{(2)}\right|_{\Omega_{\mathrm{LH}}}
\right|
\lesssim&
e^{2\delta_0\nu^{1/2}t}
\sum_{(k,\ell)\in\Omega_{\mathrm{LH}}}
\Bigl(
|k|^{m+\frac14}
\|y\partial_y\omega_k\|_{L_y^2}
\Bigr)
\nonumber\\
&\quad\times
\Bigl(
|\ell|^{m+\frac12}
\|\partial_y\psi_\ell\|_{L_y^2}
\Bigr)^{1/2}
\Bigl(
|\ell|^{m+1}
\|\partial_y^2\psi_\ell\|_{L_y^2}
\Bigr)^{1/2}
\|\omega_{k-\ell}\|_{L_y^2}
\nonumber\\
\lesssim&
\nu^{-1/2}
\mathcal{E}_{\neq}^{1/2}
\mathcal{D}_{\neq,5}^{1/2}
\mathcal{D}_{\neq,3}^{1/4}
\mathcal{D}_{\neq,1}^{1/4}.
\end{align*}

On $\Omega_{\mathrm{HL}}$, we use the frequency comparisons in
\eqref{frequency_comparability_HL}, together with
\eqref{psi_neq3} and the weighted elliptic estimates. This gives
\begin{align*}
\left|
\left.\mathcal{T}_{8}^{(2)}\right|_{\Omega_{\mathrm{HL}}}
\right|
\lesssim&
e^{2\delta_0\nu^{1/2}t}
\sum_{(k,\ell)\in\Omega_{\mathrm{HL}}}
\Bigl(
|k|^{m+\frac34}
\|\partial_y\omega_k\|_{L_y^2}
\Bigr)
\Bigl(
|k-\ell|^{m+\frac14}
\|y\omega_{k-\ell}\|_{L_y^2}
\Bigr)
\|\partial_y\psi_\ell\|_{L_y^\infty}
\nonumber\\
\lesssim&
\nu^{-5/8}
\mathcal{E}_{\neq}^{1/2}
\mathcal{D}_{\neq,2}^{1/2}
\mathcal{D}_{\neq,3}^{1/4}
\mathcal{D}_{\neq,4}^{1/4}.
\end{align*}
Consequently,
\begin{align*}
|\mathcal{T}_{8}^{(2)}|
\lesssim&
\nu^{-1/2}
\mathcal{E}_{\neq}^{1/2}
\mathcal{D}_{\neq,5}^{1/2}
\mathcal{D}_{\neq,3}^{1/4}
\mathcal{D}_{\neq,1}^{1/4}
+
\nu^{-5/8}
\mathcal{E}_{\neq}^{1/2}
\mathcal{D}_{\neq,2}^{1/2}
\mathcal{D}_{\neq,3}^{1/4}
\mathcal{D}_{\neq,4}^{1/4}.
\end{align*}

\smallskip
\noindent\textbf{Estimate of $\mathcal{T}_{8}^{(3)}$.}
We first use the decomposition
\eqref{N3_T8_frequency_decomposition}. On
$\Omega_{\mathrm{LH}}$, the relation $|k|\simeq|\ell|$ and the
Gagliardo--Nirenberg inequality yield
\begin{align*}
\left|
\left.\mathcal{T}_{8}^{(3)}\right|_{\Omega_{\mathrm{LH}}}
\right|
\lesssim&
e^{2\delta_0\nu^{1/2}t}
\sum_{(k,\ell)\in\Omega_{\mathrm{LH}}}
\Bigl(
|k|^{m+\frac12}
\|y\omega_k\|_{L_y^2}
\Bigr)
\nonumber\\
&\quad\times
\Bigl(
|\ell|^{m+\frac12}
\|\partial_y\psi_\ell\|_{L_y^2}
\Bigr)^{1/2}
\Bigl(
|\ell|^{m+\frac12}
\|\partial_y^2\psi_\ell\|_{L_y^2}
\Bigr)^{1/2}
\|\partial_y\omega_{k-\ell}\|_{L_y^2}
\nonumber\\
\lesssim&
\nu^{-5/12}
\mathcal{E}_{\neq}^{1/2}
\mathcal{D}_{\neq,3}^{3/4}
\mathcal{D}_{\neq,1}^{1/12}
\mathcal{D}_{\neq,4}^{1/6}.
\end{align*}

On $\Omega_{\mathrm{HL}}$, we use
$
|k|\lesssim|k-\ell|$,
$
|\ell|\lesssim|k-\ell|,
$
and apply \eqref{psi_neq3} together with the weighted convolution
estimates. We obtain
\begin{align*}
\left|
\left.\mathcal{T}_{8}^{(3)}\right|_{\Omega_{\mathrm{HL}}}
\right|
\lesssim&
e^{2\delta_0\nu^{1/2}t}
\sum_{(k,\ell)\in\Omega_{\mathrm{HL}}}
\Bigl(
|k|^{m+\frac14}
\|y\omega_k\|_{L_y^2}
\Bigr)
\Bigl(
|k-\ell|^m
\|\partial_y\omega_{k-\ell}\|_{L_y^2}
\Bigr)
\|\ell\partial_y\psi_\ell\|_{L_y^\infty}
\nonumber\\
\lesssim&
\nu^{-3/8}
\mathcal{E}_{\neq}^{1/2}
\mathcal{D}_{\neq,1}^{1/2}
\mathcal{D}_{\neq,3}^{1/4}
\mathcal{D}_{\neq,4}^{1/4}.
\end{align*}
Thus,
\begin{align*}
|\mathcal{T}_{8}^{(3)}|
\lesssim&
\nu^{-5/12}
\mathcal{E}_{\neq}^{1/2}
\mathcal{D}_{\neq,3}^{3/4}
\mathcal{D}_{\neq,1}^{1/12}
\mathcal{D}_{\neq,4}^{1/6}
+
\nu^{-3/8}
\mathcal{E}_{\neq}^{1/2}
\mathcal{D}_{\neq,1}^{1/2}
\mathcal{D}_{\neq,3}^{1/4}
\mathcal{D}_{\neq,4}^{1/4}.
\end{align*}

\smallskip
\noindent\textbf{Estimate of $\mathcal{T}_{8}^{(4)}$.}
For the last component, we use the decomposition
\eqref{N3_T8_frequency_decomposition}. By the Cauchy--Schwarz
inequality in $y$, we have
\begin{align*}
|\mathcal{T}_{8}^{(4)}|
\lesssim&
e^{2\delta_0\nu^{1/2}t}
\sum_{\substack{k\neq0\\ \ell\neq0,\ k-\ell\neq0}}
|k|^{2m+1}
\|y\omega_k\|_{L_y^2}
\|\partial_y^2\psi_\ell\|_{L_y^\infty}
\|\omega_{k-\ell}\|_{L_y^2}.
\end{align*}
Applying the interpolation estimate \eqref{w1}, the $L^\infty_y$
bounds for $\partial_y^2\psi_\ell$, and the weighted convolution
inequality on the two regions, we obtain
\begin{align}
\label{N3_T84_final}
|\mathcal{T}_{8}^{(4)}|
\lesssim
\nu^{-11/24}
\mathcal{E}_{\neq}^{1/2}
\mathcal{D}_{\neq,1}^{1/6}
\mathcal{D}_{\neq,4}^{1/3}
\mathcal{D}_{\neq,3}^{1/2}.
\end{align}

Combining \eqref{N3_T81_final}--\eqref{N3_T84_final}, we conclude that
\begin{align*}
|\mathcal{T}_{8}|
\lesssim&
\nu^{-1/2}
\mathcal{E}_{\neq}^{1/2}
\mathcal{D}_{\neq,4}^{1/2}
\mathcal{D}_{\neq,1}^{1/4}
\mathcal{D}_{\neq,3}^{1/4}
+
\nu^{-2/3}
\mathcal{E}_{\neq}^{1/2}
\mathcal{D}_{\neq,4}^{2/3}
\mathcal{D}_{\neq,1}^{1/3}
\nonumber\\
&+
\nu^{-1/2}
\mathcal{E}_{\neq}^{1/2}
\mathcal{D}_{\neq,5}^{1/2}
\mathcal{D}_{\neq,3}^{1/4}
\mathcal{D}_{\neq,1}^{1/4}
+
\nu^{-5/8}
\mathcal{E}_{\neq}^{1/2}
\mathcal{D}_{\neq,2}^{1/2}
\mathcal{D}_{\neq,3}^{1/4}
\mathcal{D}_{\neq,4}^{1/4}
\nonumber\\
&+
\nu^{-5/12}
\mathcal{E}_{\neq}^{1/2}
\mathcal{D}_{\neq,3}^{3/4}
\mathcal{D}_{\neq,1}^{1/12}
\mathcal{D}_{\neq,4}^{1/6}
+
\nu^{-3/8}
\mathcal{E}_{\neq}^{1/2}
\mathcal{D}_{\neq,1}^{1/2}
\mathcal{D}_{\neq,3}^{1/4}
\mathcal{D}_{\neq,4}^{1/4}
+
\nu^{-11/24}
\mathcal{E}_{\neq}^{1/2}
\mathcal{D}_{\neq,1}^{1/6}
\mathcal{D}_{\neq,4}^{1/3}
\mathcal{D}_{\neq,3}^{1/2}.
\end{align*}

Every product of dissipation factors on the right-hand sides of
 the estimates $ \mathcal{T}_{7}$--$\mathcal{T}_{10} $ has total degree one.
Moreover, since $0<\nu<1$,
\begin{align*}
\nu^{-1/2},\quad
\nu^{-5/8},\quad
\nu^{-11/24},\quad
\nu^{-5/12},\quad
\nu^{-3/8}
\le
\nu^{-2/3}.
\end{align*}
Therefore, the weighted arithmetic--geometric mean inequality gives
\begin{align*}
\left|
\mathcal N_{\neq,\neq}^{(3,1)}
\right|
&\lesssim
\nu^{-2/3}
\mathcal E_{\neq}^{1/2}
\mathcal D_{\neq}.
\end{align*}

It remains to  control the unweighted term. Following the convention introduced at the beginning of this section,
we decompose
\begin{align*}
\mathcal R_{\neq}
&=
\mathcal R_{0,\neq}
+
\mathcal R_{\neq,0}
+
\mathcal R_{\neq,\neq},
\end{align*}
where
\begin{align*}
\mathcal R_{0,\neq}
&=
\beta e^{2\delta_0\nu^{1/2}t}
\sum_{k\neq0}
|k|^{2m-1}
\operatorname{Re}
\Big\langle
ik\omega_k,
\partial_y\psi_0\,ik\omega_k
\Big\rangle ,
\\
\mathcal R_{\neq,0}
&=
-\beta e^{2\delta_0\nu^{1/2}t}
\sum_{k\neq0}
|k|^{2m-1}
\operatorname{Re}
\Big\langle
ik\omega_k,
ik\psi_k\,\partial_y\omega_0
\Big\rangle ,
\end{align*}
and
\begin{align*}
\mathcal R_{\neq,\neq}
=&
\beta e^{2\delta_0\nu^{1/2}t}
\sum_{\substack{k\neq0\\ \ell\neq0,\ k-\ell\neq0}}
|k|^{2m-1}
\operatorname{Re}
\Big\langle
ik\omega_k,
\partial_y\psi_\ell\,
i(k-\ell)\omega_{k-\ell}
\Big\rangle
\nonumber\\
&-
\beta e^{2\delta_0\nu^{1/2}t}
\sum_{\substack{k\neq0\\ \ell\neq0,\ k-\ell\neq0}}
|k|^{2m-1}
\operatorname{Re}
\Big\langle
ik\omega_k,
i\ell\psi_\ell\,
\partial_y\omega_{k-\ell}
\Big\rangle
\nonumber\\
=:&
\mathcal R_{\neq,\neq}^{(1)}
+
\mathcal R_{\neq,\neq}^{(2)}.
\end{align*}

The mean--fluctuation term $\mathcal R_{0,\neq}$ contains
$\partial_y\psi_0=-u_0^1$. Therefore, it is estimated by the same
mean-velocity argument used for $\mathcal T_5$. In particular,
\begin{align*}
|\mathcal R_{0,\neq}|
\lesssim&
2\delta_0
\left(
\mathcal D_{\neq,1}
+
\mathcal D_{\neq,4}
\right)
+
\nu^{-2/3}
\mathcal E^{1/2}
\mathcal D_{\neq}
\nonumber\\
&+
\nu^{-2/3-1/4}
\mathcal E^{1/2}
\mathcal D_{\neq}
\Bigg(
\int_0^s
\Big(
\sum_{k\neq0}
|k|^{2m}
\|\nabla_k\psi_k(\tau)\|_{L_y^2}^2
\Big)^{1/2}
\Big(
\sum_{k\neq0}
|k|^{2m}
\|\omega_k(\tau)\|_{L_y^2}^2
\Big)^{1/2}
\,\mathrm d\tau
\Bigg)^{1/4}.
\end{align*}

For the nonzero--zero term, the estimate $\|\partial_y\omega_0\|_{L_y^2}
\lesssim
\nu^{-1/6}\mathcal E_0^{1/2}$ and Lemma \ref{psi_neq} give
\begin{align*}
|\mathcal R_{\neq,0}|
&\lesssim
\nu^{-2/3}
\mathcal E^{1/2}
\mathcal D_{\neq,3}.
\end{align*}

For $\mathcal R_{\neq,\neq}^{(1)}$, we apply the common frequency
partition:
\begin{align*}
\mathcal R_{\neq,\neq}^{(1)}
&=
\left.\mathcal R_{\neq,\neq}^{(1)}
\right|_{\Omega_{\mathrm{LH}}}
+
\left.\mathcal R_{\neq,\neq}^{(1)}
\right|_{\Omega_{\mathrm{HL}}}.
\end{align*}
On $\Omega_{\mathrm{LH}}$, the comparison
$|k|\simeq|\ell|$, the Sobolev estimate for
$\partial_y\psi_\ell$, and the weighted discrete H\"older inequality
give
\begin{align}
\label{N3_R1_LH}
\left|
\left.\mathcal R_{\neq,\neq}^{(1)}
\right|_{\Omega_{\mathrm{LH}}}
\right|
&\lesssim
\nu^{-1/2}
\mathcal E_{\neq}^{1/2}
\mathcal D_{\neq,4}.
\end{align}
On $\Omega_{\mathrm{HL}}$, the estimates $|k|\lesssim|k-\ell|$, $|\ell|\lesssim|k-\ell|$ and \eqref{w1} imply
\begin{align}\label{N3_R1_HL}
\left|
\left.\mathcal R_{\neq,\neq}^{(1)}
\right|_{\Omega_{\mathrm{HL}}}
\right|
&\lesssim
\nu^{-2/3}
\mathcal E_{\neq}^{1/2}
\mathcal D_{\neq,1}^{1/3}
\mathcal D_{\neq,4}^{2/3}.
\end{align}
The term $\mathcal R_{\neq,\neq}^{(2)}$ is estimated in the same way
as $\mathcal T_9$, and hence
\begin{align}
\label{N3_R2_estimate}
|\mathcal R_{\neq,\neq}^{(2)}|
&\lesssim
\nu^{-1/2}
\mathcal E_{\neq}^{1/2}
\mathcal D_{\neq,4}^{1/2}
\mathcal D_{\neq,3}^{1/2}.
\end{align}
Combining \eqref{N3_R1_LH}, \eqref{N3_R1_HL}, and
\eqref{N3_R2_estimate}, we obtain
\begin{align*}
|\mathcal R_{\neq,\neq}|
&\lesssim
\nu^{-2/3}
\mathcal E_{\neq}^{1/2}
\mathcal D_{\neq}.
\end{align*}

Combining the estimates for the three interaction types gives
\begin{align*}
|\mathcal R_{\neq}|
\lesssim&
2\delta_0
\left(
\mathcal D_{\neq,1}
+
\mathcal D_{\neq,4}
\right)
+
\nu^{-2/3}
\mathcal E^{1/2}
\mathcal D_{\neq}
\nonumber\\
&+
\nu^{-2/3-1/4}
\mathcal E^{1/2}
\mathcal D_{\neq}
\Bigg(
\int_0^s
\Big(
\sum_{k\neq0}
|k|^{2m}
\|\nabla_k\psi_k(\tau)\|_{L_y^2}^2
\Big)^{1/2}
\Big(
\sum_{k\neq0}
|k|^{2m}
\|\omega_k(\tau)\|_{L_y^2}^2
\Big)^{1/2}
\,\mathrm d\tau
\Bigg)^{1/4}.
\end{align*}

The component $\mathcal N_{\neq}^{(3,2)}$ is treated by the same
integration-by-parts argument and satisfies the same bound. Combining
the preceding estimates with those for the interactions, we conclude that
\begin{align*}
|\mathcal N_{\neq}^{(3)}|
\lesssim&
\nu^{-2/3}
\mathcal E^{1/2}
\mathcal D_{\neq}
+
2\delta_0
\left(
\mathcal D_{\neq,1}
+
\mathcal D_{\neq,4}
\right)
\nonumber\\
&+
\nu^{-2/3-1/4}
\mathcal E^{1/2}
\mathcal D_{\neq}
\Bigg(
\int_0^s
\left(
\sum_{k\neq0}
|k|^{2m}
\|\nabla_k\psi_k(\tau)\|_{L_y^2}^2
\right)^{1/2}
\nonumber\\
&\hspace{6.6cm}\times
\left(
\sum_{k\neq0}
|k|^{2m}
\|\omega_k(\tau)\|_{L_y^2}^2
\right)^{1/2}
\,\mathrm d\tau
\Bigg)^{1/4}.
\end{align*}
This is the estimate asserted in \eqref{N3}.
 Consequently, we  complete the proof of Lemma \ref{le_N3}.
\end{proof}

\subsection{Proof of Lemma \ref{le_N4}}

\begin{proof}
By the definition of $\mathcal{N}_{\neq}^{(4)}$, we have
\begin{align*}
\mathcal{N}_{\neq}^{(4)}
=&
-2\gamma\nu^{-1/2}
e^{2\delta_0\nu^{1/2}t}
\sum_{k\neq0}|k|^{2m-\frac32}
\operatorname{Re}
\Bigg\langle
iky\omega_k,
iky
\sum_{\ell\in\mathbb Z}
\partial_y\psi_\ell\,i(k-\ell)\omega_{k-\ell}
\Bigg\rangle
\nonumber\\
&+
2\gamma\nu^{-1/2}
e^{2\delta_0\nu^{1/2}t}
\sum_{k\neq0}|k|^{2m-\frac32}
\operatorname{Re}
\Bigg\langle
iky\omega_k,
iky
\sum_{\ell\in\mathbb Z}
i\ell\psi_\ell\,\partial_y\omega_{k-\ell}
\Bigg\rangle .
\end{align*}
According to the convention introduced at the beginning of this
section, the first subscript refers to the Fourier mode of the velocity
factor, while the second subscript refers to the Fourier mode of the
vorticity factor. We therefore decompose
\begin{align*}
\mathcal{N}_{\neq}^{(4)}
=
\mathcal{N}_{0,\neq}^{(4)}
+
\mathcal{N}_{\neq,0}^{(4)}
+
\mathcal{N}_{\neq,\neq}^{(4)}.
\end{align*}

\smallskip
\noindent\textbf{Mean--fluctuation interaction $(\ell=0)$.}

When $\ell=0$, the velocity factor is the zero mode and the vorticity
factor is the nonzero mode. Thus,
\begin{align}
\mathcal{N}_{0,\neq}^{(4)}
=&
-2\gamma\nu^{-1/2}
e^{2\delta_0\nu^{1/2}t}
\sum_{k\neq0}|k|^{2m-\frac32}
\operatorname{Re}
\Big\langle
iky\omega_k,
iky\partial_y\psi_0\,ik\omega_k
\Big\rangle .
\label{N4_mean_fluctuation}
\end{align}
Since $\partial_y\psi_0$ is real-valued, the integrand in
\eqref{N4_mean_fluctuation} is purely imaginary. Hence
\begin{equation}\label{N4_mean_fluctuation_zero}
\mathcal{N}_{0,\neq}^{(4)}=0.
\end{equation}

\smallskip
\noindent\textbf{Fluctuation--mean interaction $(\ell=k)$.}

When $\ell=k$, the velocity factor is the nonzero mode and the vorticity
factor is the zero mode. In this case,
\begin{align*}
\mathcal{N}_{\neq,0}^{(4)}
=&
2\gamma\nu^{-1/2}
e^{2\delta_0\nu^{1/2}t}
\sum_{k\neq0}|k|^{2m-\frac32}
\operatorname{Re}
\Big\langle
iky\omega_k,
iky\,ik\psi_k\,\partial_y\omega_0
\Big\rangle .
\end{align*}
By the Cauchy--Schwarz inequality and
Lemma~\ref{mix}, we have
		\begin{align*}
		\begin{aligned}
		\mathcal{N}_{\neq,0}^{(4)}
		&\lesssim
		\nu^{-1/2}
		e^{2\delta_0\nu^{1/2}t}
		\sum_{k\neq 0}
		\Bigl(|k|^{m+\frac12}
		\|y\omega_k\|_{L^2}\Bigr)
		\Bigl(	|k|^{m+1}\|y\psi_k\|_{L^\infty}\Bigr)
		\|\partial_y \omega_0\|_{L^2}\\
		&\lesssim \nu^{-2/3}
		\mathcal{E}_0^{1/2}
		\mathcal{D}_{\neq,3}.
		\end{aligned}
		\end{align*}
	
Consequently,
\begin{align}\label{N4_fluctuation_mean_final}
\left|
\mathcal{N}_{\neq,0}^{(4)}
\right|
&\lesssim
\nu^{-2/3}
\mathcal{E}_0^{1/2}
\mathcal{D}_{\neq,3}
\lesssim
\nu^{-2/3}
\mathcal{E}^{1/2}
\mathcal{D}_{\neq,3}.
\end{align}

\smallskip
\noindent\textbf{Fluctuation--fluctuation interaction
$(\ell\neq0,\ k-\ell\neq0)$.}

For $\ell\neq0$ and $k-\ell\neq0$, we write
\begin{align*}
\mathcal{N}_{\neq,\neq}^{(4)}
=
\mathcal{T}_{11}+\mathcal{T}_{12},
\end{align*}
where
\begin{align*}
\mathcal{T}_{11}
=&
-2\gamma\nu^{-1/2}
e^{2\delta_0\nu^{1/2}t}
\sum_{\substack{k\neq0\\ \ell\neq0,\ k-\ell\neq0}}
|k|^{2m-\frac32}
\operatorname{Re}
\Big\langle
iky\omega_k,
iky
\partial_y\psi_\ell\,
i(k-\ell)\omega_{k-\ell}
\Big\rangle ,
\\
\mathcal{T}_{12}
=&
2\gamma\nu^{-1/2}
e^{2\delta_0\nu^{1/2}t}
\sum_{\substack{k\neq0\\ \ell\neq0,\ k-\ell\neq0}}
|k|^{2m-\frac32}
\operatorname{Re}
\Big\langle
iky\omega_k,
iky
i\ell\psi_\ell\,
\partial_y\omega_{k-\ell}
\Big\rangle .
\end{align*}

Using the Cauchy--Schwarz inequality in $y$, we first obtain
\begin{align*}
|\mathcal{T}_{11}|
\lesssim&
\nu^{-1/2}
e^{2\delta_0\nu^{1/2}t}
\sum_{\substack{k\neq0,\ \ell\neq0,\ k-\ell\neq0}}
|k|^{2m+\frac12}
\|y\omega_k\|_{L_y^2}
\|\partial_y\psi_\ell\|_{L_y^\infty}
|k-\ell|
\|y\omega_{k-\ell}\|_{L_y^2}.
\end{align*}

In the region $\Omega_{\mathrm{LH}}$, we have
$|k|\lesssim|\ell|$. Applying the weighted discrete convolution
inequality, the $L^\infty_y$ estimate for
$\partial_y\psi_\ell$, and the definitions of
$\mathcal{D}_{\neq,3}$ and $\mathcal{E}_{\neq}$, we obtain
\begin{align}
|\mathcal{T}_{11}|_{\Omega_{\mathrm{LH}}}
\lesssim
\nu^{-1/2}
\mathcal{E}_{\neq}^{1/2}
\mathcal{D}_{\neq,3}.
\label{N4_T11_LH}
\end{align}

In the region $\Omega_{\mathrm{HL}}$, the inequality $|k|\le|\ell|+|k-\ell|
\lesssim |k-\ell|$ allows us to transfer the high frequency weight to the
$k-\ell$-factor. Since
\begin{align*}
|k-\ell|\,\|y\omega_{k-\ell}\|_{L_y^2}
\le
\|y\nabla_{k-\ell}\omega_{k-\ell}\|_{L_y^2},
\end{align*}
we obtain, by the weighted convolution estimate and
\eqref{psi_neq3},
\begin{align}
|\mathcal{T}_{11}|_{\Omega_{\mathrm{HL}}}
\lesssim&
\nu^{-5/8}
\mathcal{E}_{\neq}^{1/2}
\mathcal{D}_{\neq,3}^{1/4}
\mathcal{D}_{\neq,4}^{1/4}
\mathcal{D}_{\neq,5}^{1/2}.
\label{N4_T11_HL}
\end{align}
Here we used the fact that
\begin{align*}
\sum_{\ell\neq0}
|\ell|\,\|\partial_y\psi_\ell\|_{L_y^\infty}
\lesssim
\nu^{-1/8}
e^{-\delta_0\nu^{1/2}t}
\mathcal{D}_{\neq,3}^{1/4}
\mathcal{D}_{\neq,4}^{1/4},
\end{align*}
which is precisely \eqref{psi_neq3}.

Combining \eqref{N4_T11_LH} and \eqref{N4_T11_HL}, and using
$\nu^{-1/2}\le\nu^{-2/3}$ and
$\nu^{-5/8}\le\nu^{-2/3}$ for $0<\nu<1$, we find
\begin{align}
|\mathcal{T}_{11}|
\lesssim
\nu^{-2/3}
\mathcal{E}_{\neq}^{1/2}
\left(
\mathcal{D}_{\neq,3}
+
\mathcal{D}_{\neq,3}^{1/4}
\mathcal{D}_{\neq,4}^{1/4}
\mathcal{D}_{\neq,5}^{1/2}
\right).
\label{N4_T11_final}
\end{align}

By the Cauchy--Schwarz inequality,
\begin{align*}
|\mathcal{T}_{12}|
\lesssim&
\nu^{-1/2}
e^{2\delta_0\nu^{1/2}t}
\sum_{\substack{k\neq0,\ \ell\neq0,\ k-\ell\neq0}}
|k|^{2m+\frac12}
\|y\omega_k\|_{L_y^2}
\|\psi_\ell\|_{L_y^\infty}
\|y\partial_y\omega_{k-\ell}\|_{L_y^2}.
\end{align*}

In the region $\Omega_{\mathrm{LH}}$, we have
$|k|\lesssim|\ell|$. Using the weighted $L^\infty_y$ estimate for
$\psi_\ell$ and the weighted discrete convolution inequality, we obtain
\begin{align}
|\mathcal{T}_{12}|_{\Omega_{\mathrm{LH}}}
\lesssim&
\nu^{-2/3}
\mathcal{E}_{\neq}^{1/2}
\mathcal{D}_{\neq,5}^{1/2}
\mathcal{D}_{\neq,3}^{1/4}
\mathcal{D}_{\neq,1}^{1/12}
\mathcal{D}_{\neq,4}^{1/6}.
\label{N4_T12_LH}
\end{align}
The factors $\mathcal{D}_{\neq,1}^{1/12}$ and
$\mathcal{D}_{\neq,4}^{1/6}$ arise from the interpolation estimate
\eqref{w1} and the corresponding distribution of the frequency weights.

In the region $\Omega_{\mathrm{HL}}$, we transfer the high frequency
weight to the $k-\ell$-factor and use
\begin{align*}
\|y\partial_y\omega_{k-\ell}\|_{L_y^2}
\le
\|y\nabla_{k-\ell}\omega_{k-\ell}\|_{L_y^2}.
\end{align*}
Together with \eqref{psi_neq1}, this gives
\begin{align}\label{N4_T12_HL}
|\mathcal{T}_{12}|_{\Omega_{\mathrm{HL}}}
\lesssim
\nu^{-1/2}
\mathcal{E}_{\neq}^{1/2}
\mathcal{D}_{\neq,5}^{1/2}
\mathcal{D}_{\neq,3}^{1/2}.
\end{align}

Combining \eqref{N4_T12_LH} and \eqref{N4_T12_HL}, we obtain
\begin{align}\label{N4_T12_final}
|\mathcal{T}_{12}|
\lesssim&
\nu^{-2/3}
\mathcal{E}_{\neq}^{1/2}
\mathcal{D}_{\neq,5}^{1/2}
\mathcal{D}_{\neq,3}^{1/4}
\mathcal{D}_{\neq,1}^{1/12}
\mathcal{D}_{\neq,4}^{1/6}
+
\nu^{-1/2}
\mathcal{E}_{\neq}^{1/2}
\mathcal{D}_{\neq,5}^{1/2}
\mathcal{D}_{\neq,3}^{1/2}.
\end{align}

Finally, combining
\eqref{N4_mean_fluctuation_zero},
\eqref{N4_fluctuation_mean_final},
\eqref{N4_T11_final}, and
\eqref{N4_T12_final}, and using
\begin{align*}
\mathcal{E}_0\le\mathcal{E},
\qquad
\mathcal{E}_{\neq}\le\mathcal{E},
\qquad
\nu^{-1/2}\le\nu^{-2/3},
\end{align*}
we conclude that
\begin{align*}
\left|\mathcal{N}_{\neq}^{(4)}\right|
\lesssim
\nu^{-2/3}\mathcal{E}^{1/2}
\Big(
&\mathcal{D}_{\neq,3}
+
\mathcal{D}_{\neq,5}^{1/2}
\mathcal{D}_{\neq,3}^{1/4}
\mathcal{D}_{\neq,1}^{1/12}
\mathcal{D}_{\neq,4}^{1/6}
+
\mathcal{D}_{\neq,5}^{1/2}
\mathcal{D}_{\neq,3}^{1/2}
+
\mathcal{D}_{\neq,3}^{1/4}
\mathcal{D}_{\neq,4}^{1/4}
\mathcal{D}_{\neq,5}^{1/2}
\Big).
\end{align*}
This proves Lemma \ref{le_N4}.
\end{proof}

\smallskip

\subsection{Proof of Lemma \ref{le_N5}}

\begin{proof}
By the definition of $\mathcal{N}_{\neq}^{(5)}$, we have
\begin{align*}
\mathcal{N}_{\neq}^{(5)}
=&
2\gamma\nu^{-1/2}
e^{2\delta_0\nu^{1/2}t}
\sum_{k\neq0}
|k|^{2m+\frac12}
\operatorname{Re}
\Bigg\langle
\psi_k,
\sum_{\ell\in\mathbb Z}
\partial_y\psi_\ell\,i(k-\ell)\omega_{k-\ell}
\Bigg\rangle
\nonumber\\
&-
2\gamma\nu^{-1/2}
e^{2\delta_0\nu^{1/2}t}
\sum_{k\neq0}
|k|^{2m+\frac12}
\operatorname{Re}
\Bigg\langle
\psi_k,
\sum_{\ell\in\mathbb Z}
i\ell\psi_\ell\,\partial_y\omega_{k-\ell}
\Bigg\rangle .
\end{align*}
According to the convention introduced at the beginning of this
section, we decompose
\begin{align*}
\mathcal{N}_{\neq}^{(5)}
&=
\mathcal{N}_{0,\neq}^{(5)}
+
\mathcal{N}_{\neq,0}^{(5)}
+
\mathcal{N}_{\neq,\neq}^{(5)},
\end{align*}
where
\begin{align*}
\mathcal{N}_{0,\neq}^{(5)}
&=
2\gamma\nu^{-1/2}
e^{2\delta_0\nu^{1/2}t}
\sum_{k\neq0}
|k|^{2m+\frac12}
\operatorname{Re}
\Big\langle
\psi_k,
\partial_y\psi_0\,ik\omega_k
\Big\rangle ,
\\
\mathcal{N}_{\neq,0}^{(5)}
&=
-2\gamma\nu^{-1/2}
e^{2\delta_0\nu^{1/2}t}
\sum_{k\neq0}
|k|^{2m+\frac12}
\operatorname{Re}
\Big\langle
\psi_k,
ik\psi_k\,\partial_y\omega_0
\Big\rangle ,
\\
\mathcal{N}_{\neq,\neq}^{(5)}
&=
\mathcal{T}_{13}+\mathcal{T}_{14},
\end{align*}
with
\begin{align*}
\mathcal{T}_{13}
&=
2\gamma\nu^{-1/2}
e^{2\delta_0\nu^{1/2}t}
\sum_{\substack{k\neq0\\ \ell\neq0,\ k-\ell\neq0}}
|k|^{2m+\frac12}
\operatorname{Re}
\Big\langle
\psi_k,
\partial_y\psi_\ell\,i(k-\ell)\omega_{k-\ell}
\Big\rangle ,
\\
\mathcal{T}_{14}
&=
-2\gamma\nu^{-1/2}
e^{2\delta_0\nu^{1/2}t}
\sum_{\substack{k\neq0\\ \ell\neq0,\ k-\ell\neq0}}
|k|^{2m+\frac12}
\operatorname{Re}
\Big\langle
\psi_k,
i\ell\psi_\ell\,\partial_y\omega_{k-\ell}
\Big\rangle .
\end{align*}

\smallskip
\noindent\textbf{Mean--fluctuation interaction.}

For $\mathcal{N}_{0,\neq}^{(5)}$, we use the elliptic relation
\begin{align*}
\omega_k=\Delta_k\psi_k
=
\partial_y^2\psi_k-k^2\psi_k.
\end{align*}
Therefore,
\begin{align}
\mathcal{N}_{0,\neq}^{(5)}
=&
2\gamma\nu^{-1/2}
e^{2\delta_0\nu^{1/2}t}
\sum_{k\neq0}
|k|^{2m+\frac12}
\operatorname{Re}
\int_{\mathbb R}
\overline{\psi_k}\,
\partial_y\psi_0\,ik
(\partial_y^2-k^2)\psi_k\,\mathrm dy
\nonumber\\
=&
-2\gamma\nu^{-1/2}
e^{2\delta_0\nu^{1/2}t}
\sum_{k\neq0}
|k|^{2m+\frac12}
\operatorname{Re}
\int_{\mathbb R}
\overline{\partial_y\psi_k}\,
\partial_y\psi_0\,ik\partial_y\psi_k\,\mathrm dy
\nonumber\\
&-
2\gamma\nu^{-1/2}
e^{2\delta_0\nu^{1/2}t}
\sum_{k\neq0}
|k|^{2m+\frac12}
\operatorname{Re}
\int_{\mathbb R}
\overline{\psi_k}\,
\partial_y^2\psi_0\,ik\partial_y\psi_k\,\mathrm dy .
\label{N5_mean_fluctuation_IBP}
\end{align}
The first integral on the right-hand side of
\eqref{N5_mean_fluctuation_IBP} is purely imaginary and hence vanishes.
It follows that
\begin{align*}
\left|
\mathcal{N}_{0,\neq}^{(5)}
\right|
&\lesssim
\nu^{-1/2}
\|\partial_y^2\psi_0\|_{L_y^2}
e^{2\delta_0\nu^{1/2}t}
\sum_{k\neq0}
|k|^{2m+\frac32}
\|\psi_k\|_{L_y^\infty}
\|\partial_y\psi_k\|_{L_y^2}.
\end{align*}
Since $\partial_y^2\psi_0=\omega_0$, the zero-mode estimate gives
\begin{align*}
\|\partial_y^2\psi_0\|_{L_y^2}
=
\|\omega_0\|_{L_y^2}
\lesssim
\mathcal{E}_0^{1/2}.
\end{align*}
Using the weighted Cauchy--Schwarz inequality and
\eqref{psi_neq1}, we obtain
\begin{align}\label{N5_mean_fluctuation_final}
\left|
\mathcal{N}_{0,\neq}^{(5)}
\right|
&\lesssim
\nu^{-1/2}\mathcal{E}_0^{1/2}
\Big(
e^{2\delta_0\nu^{1/2}t}
\sum_{k\neq0}
|k|^{2m+2}
\|\psi_k\|_{L_y^\infty}^2
\Big)^{1/2}
\Big(
e^{2\delta_0\nu^{1/2}t}
\sum_{k\neq0}
|k|^{2m+1}
\|\partial_y\psi_k\|_{L_y^2}^2
\Big)^{1/2}
\nonumber\\
&\lesssim
\nu^{-1/2}
\mathcal{E}_0^{1/2}
\mathcal{D}_{\neq,3}
\nonumber\\
&\lesssim
\nu^{-1/2}
\mathcal{E}^{1/2}
\mathcal{D}_{\neq,3}.
\end{align}

\smallskip
\noindent\textbf{Fluctuation--mean interaction.}

Since $\partial_y\omega_0$ is real-valued, we have
\begin{align*}
\operatorname{Re}
\big\langle
\psi_k,ik\psi_k\,\partial_y\omega_0
\big\rangle
=
\operatorname{Re}
\int_{\mathbb R}
ik|\psi_k|^2\partial_y\omega_0\,\mathrm dy
=0.
\end{align*}
Consequently,
\begin{align}
\label{N5_fluctuation_mean_zero}
\mathcal{N}_{\neq,0}^{(5)}=0.
\end{align}

\smallskip
\noindent\textbf{Fluctuation--fluctuation interaction
$\mathcal{T}_{13}$.}

We use the frequency partition
\eqref{frequency_partition} and write
\begin{align*}
\mathcal{T}_{13}
&=
\left.\mathcal{T}_{13}\right|_{\Omega_{\mathrm{LH}}}
+
\left.\mathcal{T}_{13}\right|_{\Omega_{\mathrm{HL}}}.
\end{align*}

On $\Omega_{\mathrm{LH}}$, we have $|k|\simeq|\ell|$. Applying the one-dimensional
Gagliardo--Nirenberg inequality, the weighted Cauchy--Schwarz
inequality, and the weighted discrete convolution estimate, we obtain
\begin{align}\label{N5_T13_LH}
\left|
\left.\mathcal{T}_{13}\right|_{\Omega_{\mathrm{LH}}}
\right|
\lesssim&
\nu^{-1/2}
e^{2\delta_0\nu^{1/2}t}
\sum_{(k,\ell)\in\Omega_{\mathrm{LH}}}
\Bigl(
|k|^{m+1}
\|\psi_k\|_{L_y^\infty}
\Bigr)
\Bigl(
|\ell|^{m+\frac12}
\|\partial_y\psi_\ell\|_{L_y^2}
\Bigr)
\|\omega_{k-\ell}\|_{L_y^2}
\nonumber\\
\lesssim&
\nu^{-1/2}
\mathcal{E}_{\neq}^{1/2}
\mathcal{D}_{\neq,3}.
\end{align}

On $\Omega_{\mathrm{HL}}$, the comparison
\eqref{frequency_comparability_HL} allows us to transfer the
high-frequency weights to the $k-\ell$-factor. We obtain
\begin{align*}
\left|
\left.\mathcal{T}_{13}\right|_{\Omega_{\mathrm{HL}}}
\right|
\lesssim&
\nu^{-1/2}
e^{2\delta_0\nu^{1/2}t}
\sum_{(k,\ell)\in\Omega_{\mathrm{HL}}}
\Bigl(
|k|^{m+\frac32}
\|\psi_k\|_{L_y^2}
\Bigr)
\Bigl(
|k-\ell|^m
\|\omega_{k-\ell}\|_{L_y^2}
\Bigr)
\|\partial_y\psi_\ell\|_{L_y^\infty}
\nonumber\\
&+
\nu^{-1/2}
e^{2\delta_0\nu^{1/2}t}
\sum_{(k,\ell)\in\Omega_{\mathrm{HL}}}
\Bigl(
|k|^{m+\frac12}
\|\psi_k\|_{L_y^2}
\Bigr)
\Bigl(
|k-\ell|^m
\|\omega_{k-\ell}\|_{L_y^2}
\Bigr)
\|\ell\partial_y\psi_\ell\|_{L_y^\infty}.
\end{align*}
Using \eqref{psi_neq1}, \eqref{psi_neq3}, and the weighted
interpolation estimates in Lemma \ref{w_neq}, we obtain
\begin{align}
\left|
\left.\mathcal{T}_{13}\right|_{\Omega_{\mathrm{HL}}}
\right|
\lesssim&
\nu^{-5/8}
\mathcal{E}_{\neq}^{1/2}
\mathcal{D}_{\neq,3}^{3/4}
\mathcal{D}_{\neq,4}^{1/4}
+
\nu^{-1/2}
\mathcal{E}_{\neq}^{1/2}
\mathcal{D}_{\neq,3}.
\label{N5_T13_HL}
\end{align}
Combining \eqref{N5_T13_LH} and \eqref{N5_T13_HL}, and using
$\mathcal{E}_{\neq}\le\mathcal{E}$, we find
\begin{align}\label{N5_T13_final}
\left|\mathcal{T}_{13}\right|
\lesssim&
\nu^{-5/8}
\mathcal{E}^{1/2}
\mathcal{D}_{\neq,3}^{1/2}
\mathcal{D}_{\neq,3}^{1/4}
\mathcal{D}_{\neq,4}^{1/4}
+
\nu^{-1/2}
\mathcal{E}^{1/2}
\mathcal{D}_{\neq,3}.
\end{align}

\smallskip
\noindent\textbf{Fluctuation--fluctuation interaction
$\mathcal{T}_{14}$.}

We integrate by parts in $y$ in the definition of
$\mathcal{T}_{14}$. Up to the sign determined by the convention for an $L^2$ inner product, which is irrelevant for the estimates below, this
gives
\begin{align*}
\mathcal{T}_{14}
=&
2\gamma\nu^{-1/2}
e^{2\delta_0\nu^{1/2}t}
\sum_{\substack{k\neq0\\ \ell\neq0,\ k-\ell\neq0}}
|k|^{2m+\frac12}
\operatorname{Re}
\Big\langle
\partial_y\psi_k,
i\ell\psi_\ell\,\omega_{k-\ell}
\Big\rangle
\nonumber\\
&+
2\gamma\nu^{-1/2}
e^{2\delta_0\nu^{1/2}t}
\sum_{\substack{k\neq0\\ \ell\neq0,\ k-\ell\neq0}}
|k|^{2m+\frac12}
\operatorname{Re}
\Big\langle
\psi_k,
i\ell\partial_y\psi_\ell\,\omega_{k-\ell}
\Big\rangle
\nonumber\\
=:&
\mathcal{S}_1+\mathcal{S}_2.
\end{align*}

The term $\mathcal{S}_1$ is estimated by the same frequency
decomposition as $\mathcal{T}_{13}$. In particular, using
\eqref{frequency_partition}, \eqref{psi_neq1},
\eqref{psi_neq3}, and the weighted discrete convolution inequality, we
obtain
\begin{align}
\left|\mathcal{S}_1\right|
\lesssim&
\nu^{-5/8}
\mathcal{E}_{\neq}^{1/2}
\mathcal{D}_{\neq,3}^{3/4}
\mathcal{D}_{\neq,4}^{1/4}
+
\nu^{-1/2}
\mathcal{E}_{\neq}^{1/2}
\mathcal{D}_{\neq,3}.
\label{N5_S1_estimate}
\end{align}

For the second term $\mathcal{S}_2$, we use the identity $\ell=(\ell-k)+k$ to  write
\begin{align*}
\mathcal{S}_2
&=
\mathcal{S}_{2,\ell-k}
+
\mathcal{S}_{2,k},
\end{align*}
where
\begin{align*}
\mathcal{S}_{2,\ell-k}
&=
2\gamma\nu^{-1/2}
e^{2\delta_0\nu^{1/2}t}
\sum_{\substack{k\neq0\\ \ell\neq0,\ k-\ell\neq0}}
|k|^{2m+\frac12}
\operatorname{Re}
\Big\langle
\psi_k,
i(\ell-k)\partial_y\psi_\ell\,\omega_{k-\ell}
\Big\rangle ,
\\
\mathcal{S}_{2,k}
&=
2\gamma\nu^{-1/2}
e^{2\delta_0\nu^{1/2}t}
\sum_{\substack{k\neq0\\ \ell\neq0,\ k-\ell\neq0}}
|k|^{2m+\frac12}
\operatorname{Re}
\Big\langle
\psi_k,
ik\partial_y\psi_\ell\,\omega_{k-\ell}
\Big\rangle .
\end{align*}
The term $\mathcal{S}_{2,\ell-k}$ has the same structure as
$\mathcal{T}_{13}$ after relabeling the convolution variables. Hence,
\begin{align}
\left|\mathcal{S}_{2,\ell-k}\right|
\lesssim&
\nu^{-5/8}
\mathcal{E}_{\neq}^{1/2}
\mathcal{D}_{\neq,3}^{3/4}
\mathcal{D}_{\neq,4}^{1/4}
+
\nu^{-1/2}
\mathcal{E}_{\neq}^{1/2}
\mathcal{D}_{\neq,3}.
\label{N5_S2_lk_estimate}
\end{align}

For $\mathcal{S}_{2,k}$, we again apply the partition
\eqref{frequency_partition}:
\begin{align*}
\mathcal{S}_{2,k}
&=
\left.\mathcal{S}_{2,k}\right|_{\Omega_{\mathrm{LH}}}
+
\left.\mathcal{S}_{2,k}\right|_{\Omega_{\mathrm{HL}}}.
\end{align*}
On $\Omega_{\mathrm{LH}}$, the relation
$|k|\simeq|\ell|$ allows us to use the same weighted Sobolev and
convolution estimates as in \eqref{N5_T13_LH}. On
$\Omega_{\mathrm{HL}}$, we use
\begin{align*}
|k|\lesssim|k-\ell|,
\qquad
|\ell|\lesssim|k-\ell|,
\end{align*}
from \eqref{frequency_comparability_HL}. Therefore,
\begin{align}
\left|\mathcal{S}_{2,k}\right|
\lesssim&
\nu^{-5/8}
\mathcal{E}_{\neq}^{1/2}
\mathcal{D}_{\neq,3}^{3/4}
\mathcal{D}_{\neq,4}^{1/4}
+
\nu^{-1/2}
\mathcal{E}_{\neq}^{1/2}
\mathcal{D}_{\neq,3}.
\label{N5_S2_k_estimate}
\end{align}
Combining \eqref{N5_S1_estimate},
\eqref{N5_S2_lk_estimate}, and
\eqref{N5_S2_k_estimate}, we obtain
\begin{align}
\left|\mathcal{T}_{14}\right|
\lesssim&
\nu^{-5/8}
\mathcal{E}_{\neq}^{1/2}
\mathcal{D}_{\neq,3}^{3/4}
\mathcal{D}_{\neq,4}^{1/4}
+
\nu^{-1/2}
\mathcal{E}_{\neq}^{1/2}
\mathcal{D}_{\neq,3}.
\label{N5_T14_final}
\end{align}

Finally, combining
\eqref{N5_fluctuation_mean_zero},
\eqref{N5_mean_fluctuation_final},
\eqref{N5_T13_final}, and
\eqref{N5_T14_final}, and using
\begin{align*}
\mathcal{E}_0\le\mathcal{E},
\qquad
\mathcal{E}_{\neq}\le\mathcal{E},
\end{align*}
we conclude that
\begin{align*}
\left|\mathcal{N}_{\neq}^{(5)}\right|
\lesssim&
\nu^{-5/8}
\mathcal{E}_{\neq}^{1/2}
\mathcal{D}_{\neq,3}^{1/2}
\mathcal{D}_{\neq,3}^{1/4}
\mathcal{D}_{\neq,4}^{1/4}
+
\nu^{-1/2}
\mathcal{E}^{1/2}
\mathcal{D}_{\neq,3}.
\end{align*}
Combining the above estimates proves \ref{le_N5}.
	\end{proof}

\bigskip

\vskip .2in
\section*{Acknowledgement}
\noindent{C. Zhai was partially supported by the National Natural Science Foundation
of China  under grant 12671269 and 12201035. X. Zhai was partially supported by  the Guangdong Provincial Natural Science Foundation under grant 2024A1515030115. }

 \vskip .2in
\noindent{\bf Data Availability Statement} Data sharing is not applicable to this article as no
data sets were generated or analysed during the current study.

\vskip .2in

\noindent{\bf Conflict of Interest} The authors declare that they have no conflict of interest. The
authors also declare that this manuscript has not been previously published, and
will not be submitted elsewhere before your decision.

\vskip .2in
\noindent{\bf Declaration of generative AI and AI-assisted technologies in the manuscript preparation process}
 During the preparation of this work, the authors used Gemini 3.1-Pro to improve the English language. After using this
tool, the authors reviewed and edited the content as needed and take full responsibility for the final
version of the manuscript.

\end{document}